\documentclass[10pt]{article}

\usepackage{amssymb}
\usepackage{amsmath}
\usepackage{theorem}
\usepackage{epsfig}
\usepackage{verbatim}
\usepackage{graphicx}
\usepackage{subfigure}
\usepackage{epstopdf}

\usepackage{color}

\newtheorem{theorem}{Theorem}[section]
\newtheorem{lemma}[theorem]{Lemma}

\newtheorem{prop}[theorem]{Proposition}
\newtheorem{corollary}[theorem]{Corollary}

\newtheorem{rem}[theorem]{Remark}

\newcommand{\hypg}{{}_2 F_1}

\newcommand{\al}{\alpha}

\newcommand{\ep}{\varepsilon}

\newcommand{\pa}{\partial}

\newenvironment{proof}{\begin{trivlist} \item[] {\em Proof:}}{\hfill $\Box$
                       \end{trivlist}}

\makeatletter
\renewcommand*\l@section{\@dottedtocline{1}{0em}{1.5em}}
\renewcommand*\l@subsection{\@dottedtocline{2}{1.5em}{2.3em}}
\renewcommand*\l@subsubsection{\@dottedtocline{3}{3.8em}{3.7em}}
\makeatother

\numberwithin{equation}{section}

\allowdisplaybreaks[4]

\begin{document}

\title{On the existence of stationary patches}

\author{Javier G\'omez-Serrano}

\maketitle

\begin{abstract}

In this paper, we show the existence of a family of analytic stationary patch solutions of the SQG and gSQG equations. This answers an open problem in [F. de la Hoz, Z. Hassainia, T. Hmidi. \textit{Arch. Ration. Mech. Anal.}, 220(3):1209-1281, 2016].

\vskip 0.3cm


\end{abstract}


\section{Introduction}

In this paper, we consider the generalized surface-quasi\-geo\-stro\-phic equations (gSQG):
\begin{align}\label{Int1}
\left\{ \begin{array}{ll}
\partial_{t}\theta+u\cdot\nabla\theta=0,\quad(x,t)\in\mathbb{R}^2\times\mathbb{R}_+, &\\
u=-\nabla^\perp(-\Delta)^{-1+\frac{\alpha}{2}}\theta,\\
\theta_{|t=0}=\theta_0,
\end{array} \right.
\end{align}
where $\alpha \in (0,2)$. The case $\alpha = 1$ corresponds to the surface quasi-geostrophic (SQG) equation and the limiting case $\alpha = 0$ refers to the 2D incompressible Euler equation. The case $\alpha = 2$ produces stationary solutions.

The pioneering articles of Constantin--Majda--Tabak \cite{Constantin-Majda-Tabak:formation-fronts-qg} and Held--Pierrehumbert--Garner--Swanson \cite{Held-Pierrehumbert-Garner-Swanson:sqg-dynamics} motivated the study of the SQG $(\al = 1)$ from a mathematical point of view. Since then, a lot of effort has been devoted to understanding these equations: the problem of whether the gSQG system presents global solutions or not is yet not completely understood.

The existence of weak solutions starts with the work of Resnick \cite{Resnick:phd-thesis-sqg-chicago}, where he proves the existence of global weak solutions in $L^2$ in the SQG case $\al = 1$.  In bounded domains, Constantin--Nguyen and Nguyen \cite{Constantin-Nguyen:global-weak-solutions-sqg-bounded-domains,Nguyen:global-weak-solutions-sqg-bounded-domains} proved that the same results hold. Buckmaster--Shkoller--Vicol \cite{Buckmaster-Shkoller-Vicol:nonuniqueness-sqg} have shown non-uniqueness of weak solutions for the SQG equation in certain spaces less regular than $L^{2}$. See also \cite{Marchand:existence-regularity-weak-solutions-sqg}, \cite{Chae-Constantin-Cordoba-Gancedo-Wu:gsqg-singular-velocities}, and \cite{Nahmod-Pavlovic-Staffilani-Totz:global-invariant-measures-gsqg} for more general classes of weak solutions.

In this paper, we will focus on a particular class of weak solutions, the so-called $\alpha$-\textit{patches}, which are solutions for which $\theta$ is  a step function
\begin{align}\label{pa1}
  \theta(x,t) =
  \left\{
 \begin{array}{ll}
   \theta_1, \text{ if } \ \ x \in \Omega(t) \\ 
   \theta_2, \text{ if }  \ \ x \in \Omega(t) ^c. \\
   \end{array}
  \right.
\end{align}
where $\Omega(0)\subset\mathbb{R}^2$ is a regular set given by the initial distribution of $\theta$, $\theta_1$ and $\theta_2$ are constants, and $\Omega(t)$ is the evolution of $\Omega(0)$ under the velocity field $u$. 

In this setting, local existence of patch solutions has been obtained by Rodrigo \cite{Rodrigo:evolution-sharp-fronts-qg} (for a $C^{\infty}$ boundary $\pa \Omega(0)$ in the case $\alpha = 1$), Gancedo \cite{Gancedo:existence-alpha-patch-sobolev} (for Sobolev regularity and $0 < \alpha \leq 1$) and Chae--Constantin--Cordoba--Gancedo--Wu \cite{Chae-Constantin-Cordoba-Gancedo-Wu:gsqg-singular-velocities} in the more singular case $1 < \alpha < 2$. Uniqueness for the patch equations was proved for $0 < \al < 1$ by Kiselev--Yao--Zlatos \cite{Kiselev-Yao-Zlatos:local-regularity-sqg-patch-boundary} and for $\al = 1$ by C\'ordoba--C\'ordoba--Gancedo \cite{Cordoba-Cordoba-Gancedo:uniqueness-sqg-patch}. Garra \cite{Garra:confinement-modified-sqg} obtained estimates of the growth of the support of the patch in time for $0 < \al < 1$. See also \cite{Hunter-Shu:regularized-approximate-model-sqg,Hunter-Shu-Zhang:local-wellposedness-model-sqg} for local existence results of cubic models of the $\alpha$-patch problem in the range $0 < \alpha \leq 1$.

Several authors have done numerical simulations suggesting finite time singularities. There are two scenarios: the first one (done by C\'ordoba--Fontelos--Mancho--Rodrigo \cite{Cordoba-Fontelos-Mancho-Rodrigo:evidence-singularities-contour-dynamics}), starting from two patches, suggests an asymptotically self-similar collapse between the two patches, and at the same time a blowup of the curvature at the touching point; the second one (by Scott--Dritschel \cite{Scott-Dritschel:self-similar-sqg}) evolves a thin elliptical patch and indicates a self-similar filamentation cascade ending at a singularity with a blowup of the curvature. This is consistent with the rule out of splash singularities by Gancedo--Strain \cite{Gancedo-Strain:absence-splash-muskat-SQG}. In the case with boundaries (more concretely on the halfspace), Kiselev--Ryzhik--Yao--Zlatos \cite{Kiselev-Ryzhik-Yao-Zlatos:singularity-alpha-patch-boundary} proved the formation of finite time singularities for certain patches that touch the boundary at all times.

Very little is known concerning nontrivial global solutions for the gSQG equations. C\'ordoba--G\'omez-Serrano--Ionescu \cite{Cordoba-GomezSerrano-Ionescu:global-generalized-sqg-patch} proved a generic global existence result for small solutions in the case $1 < \al < 2$, with initial data $\pa \Omega(0)$ close to the halfplane. 

Another perspective is to look for uniformly rotating solutions. These solutions are known as V-states. Deem--Zabusky \cite{Deem-Zabusky:vortex-waves-stationary} investigated this problem numerically and found the first set of families bifurcating from disks. Since then, there has been work by other authors improving the methods and computing larger classes (see for example \cite{Wu-Overman-Zabusky:steady-state-Euler-2d,Elcrat-Fornberg-Miller:stability-vortices-cylinder,LuzzattoFegiz-Williamson:efficient-numerical-method-steady-uniform-vortices,Saffman-Szeto:equilibrium-shapes-equal-uniform-vortices}).

Bifurcating from disks, Hassainia--Hmidi \cite{Hassainia-Hmidi:v-states-generalized-sqg} proved the existence of V-states with $C^k$ boundary regularity in the case $0 < \alpha < 1$. In \cite{Castro-Cordoba-GomezSerrano:existence-regularity-vstates-gsqg}, Castro--C\'ordoba--G\'omez-Serrano showed existence and $C^{\infty}$ regularity of convex global rotating solutions for the remaining open cases: $\al \in [1,2)$ for the existence, $\al \in (0,2)$ for the regularity. This boundary regularity was subsequently improved to analytic in \cite{Castro-Cordoba-GomezSerrano:analytic-vstates-ellipses}. See also \cite{Hmidi-Mateu:existence-corotating-counter-rotating} for another family of rotating solutions.

Another scenario that has been investigated is the doubly connected case. Bifurcating from annuli, de la Hoz--Hassainia--Hmidi \cite{delaHoz-Hassainia-Hmidi:doubly-connected-vstates-gsqg} established the existence of doubly connected $C^k$ V-states for $0 < \alpha < 1$, and Renault \cite{Renault:relative-equlibria-holes-sqg} proved their existence for $\alpha = 1$ in the analytic setting. In their paper, de la Hoz--Hassainia--Hmidi perform numerical simulations that suggest the existence of certain V-states with zero angular velocity and pose the question of establishing analytically the existence of stationary V-states (cf. \cite[p.1213, Remark 2]{delaHoz-Hassainia-Hmidi:doubly-connected-vstates-gsqg}).

Our goal in this paper is to solve this open question, and prove the existence of stationary patches of the gSQG equation for all $0 < \alpha < 2$. To our knowledge, this is the first nontrivial construction of stationary solutions for any $\alpha$. 

The main difficulty is that even if one could find an annulus from which bifurcate at $\Omega^{*} = 0$ using the previous ideas, there is no control on the branch and it is not clear if the continuation of the branch would intersect $\Omega = 0$ at a nontrivial point or only at the bifurcation one (which is an annulus). Another possibility is to study the local behaviour of the branch close to a bifurcation point of sufficiently small $\Omega^{*}$. However, this approach would require a nontrivial quantitative (or uniform in $\Omega^*$) control of the neighbourhoods in which the local approximation is accurate. In order to circumvent these issues, we impose stationarity and look for a different parameter in which perform the bifurcation analysis. In our case, this will be the inner radius of the annulus $b$. Specifically, we will find that for every $m \geq 2$, there exists a certain radius $b_{m}^*$ at which nontrivial stationary $m$-fold solutions bifurcate from the annulus. The precise theorem is stated in Theorem \ref{teoremaestacionarias} below. This choice of the parameter leads to a nontrivial spectral analysis in which one has to check carefully all the conditions from the Crandall-Rabinowitz \cite{Crandall-Rabinowitz:bifurcation-simple-eigenvalues} theorem.

From now on, we will assume that $\theta_2 - \theta_1 = 1$.

\subsection{The equations}

 The evolution equation for the interface of an annular $\alpha-$patch, which we parametrize as two $2 \pi$ periodic curves $Z(x)$ (outer boundary) and $z(x)$ (inner), can be written as
 \begin{align}
\label{Ecuacion-alpha-patch}
\partial_t Z(x,t) \cdot \pa_{x}^{\perp}Z(x,t) & = \left(-S(Z,Z) + S(z,Z)\right) \cdot \pa_{x}^{\perp}Z(x,t) \\
\partial_t z(x,t) \cdot \pa_{x}^{\perp}z(x,t)&  = \left(-S(Z,z) + S(z,z)\right) \cdot \pa_{x}^{\perp}z(x,t) \\
S(p,q) & = c_{\al} \int_{0}^{2\pi} \frac{\pa_{x} p(x-y) - \pa_{x} q(x)}{|p(x-y) - q(x)|^{\al}}dy,
 \end{align}

where the normalizing constant $c_{\al}$ is given by:

\begin{align*}
c_{\al} = \frac{1}{2\pi} \frac{\Gamma\left(\frac{\al}{2}\right)}{2^{1-\al}\Gamma\left(\frac{2-\al}{2}\right)}.
\end{align*}

Let $z(x,t) = (b+r(x,t))(\cos(x),\sin(x)), Z(x,t) = (1+R(x,t))(\cos(x),\sin(x))$ be the inner and outer boundaries of the patch respectively, where $b$ is a constant. Imposing stationarity, we are left to solve the following system for $(r,R) \equiv (r(x),R(x))$ and $b$:

\begin{align}\label{funcionalestacionarias}
0 = F^{1}(b,R,r) & = T_{1}(1+R) + T_{2}(b+r,1+R) \nonumber \\
0 = F^{2}(b,R,r) & = -T_{2}(1+R,b+r) - T_{1}(b+r),
\end{align}

where

\begin{align*}
T_{1}(u) &= c_{\al} \int_{0}^{2\pi} \frac{\cos(x-y)(u'(y)u(x)-u(y)u'(x))}{\left(u(x)^2+u(y)^2-2u(x)u(y)\cos(x-y)\right)^{\al/2}}dy \\
& +c_{\al} \int_{0}^{2\pi} \frac{\sin(x-y)(u(x)u(y)+u'(x)u'(y)}{\left(u(x)^2+u(y)^2-2u(x)u(y)\cos(x-y)\right)^{\al/2}}dy \\
T_{2}(p,q) & = c_{\al} \int_{0}^{2\pi} \frac{\cos(x-y)(p(y)q'(x)-p'(y)q(x))}{\left(q(x)^2+p(y)^2-2q(x)p(y)\cos(x-y)\right)^{\al/2}}dy \\
& -c_{\al} \int_{0}^{2\pi} \frac{\sin(x-y)(p(y)q(x)+p'(y)q'(x))}{\left(q(x)^2+p(y)^2-2q(x)p(y)\cos(x-y)\right)^{\al/2}}dy \\
\end{align*}

We remark that the case $r = R = 0$ corresponds to an annulus of radii $b$ and 1, yielding a stationary (though trivial) solution for any $0 < b < 1$.

\subsection{Functional spaces}

We refer to the space of analytic functions in the strip $|\Im(z)| \leq c$ as $\mathcal{C}_{w}(c)$. In our proofs, we will use the following analytic spaces. For $k \in \mathbb{Z}$:

\begin{align*}
X^{k}_{c} & = \left\{f(x) \in \mathcal{C}_{w}(c), f(x) = \sum_{j=1}^{\infty}a_{j} \cos(jx), \sum_{\pm}\int |f(x \pm ic)|^{2}dx + \sum_{\pm}\int |\pa^{k} f(x\pm ic)|^{2}dx < \infty\right\} \\
X^{k,m}_{c} & = \left\{f(x) \in \mathcal{C}_{w}(c), f(x) = \sum_{j \geq 1}^{\infty}a_{jm} \cos(jmx), \sum_{\pm}\int |f(x \pm ic)|^{2}dx + \sum_{\pm}\int |\pa^{k} f(x\pm ic)|^{2}dx < \infty\right\} \\
Y^{k}_{c} & = \left\{f(x) \in \mathcal{C}_{w}(c), f(x) = \sum_{j=1}^{\infty}a_{j} \sin(jx), \sum_{\pm}\int |f(x \pm ic)|^{2}dx + \sum_{\pm}\int |\pa^{k} f(x\pm ic)|^{2}dx < \infty\right\} \\
Y^{k,m}_{c} & = \left\{f(x) \in \mathcal{C}_{w}(c), f(x) = \sum_{j \geq 1}^{\infty}a_{jm} \sin(jmx), \sum_{\pm}\int |f(x \pm ic)|^{2}dx + \sum_{\pm}\int |\pa^{k} f(x\pm ic)|^{2}dx < \infty\right\} \\
X^{k+\al}_{c} & = \left\{f(x) \in \mathcal{C}_{w}(c), f(x) = \sum_{j=1}^{\infty}a_{j} \cos(jx), \sum_{\pm}\int |f(x \pm ic)|^{2}dx + \sum_{\pm}\int |\pa^{k} f(x\pm ic)|^{2}dx \right. \\
& + \left.\sum_{\pm} \left\|\int_{\mathbb{T}} \frac{\pa^{k}f(x\pm ic -y)-\pa^{k}f(x \pm ic)}{\left|\sin\left(\frac{y}{2}\right)\right|^{1+\al}}dy\right\|_{L^2(x)} < \infty\right\}, \quad \al \in (0,1) \\
X^{k+\al,m}_{c} & = \left\{f(x) \in \mathcal{C}_{w}(c), f(x) = \sum_{j=1}^{\infty}a_{jm} \cos(jmx), \sum_{\pm}\int |f(x \pm ic)|^{2}dx + \sum_{\pm}\int |\pa^{k} f(x\pm ic)|^{2}dx \right. \\
& + \left.\sum_{\pm} \left\|\int_{\mathbb{T}} \frac{\pa^{k}f(x\pm ic -y)-\pa^{k}f(x \pm ic)}{\left|\sin\left(\frac{y}{2}\right)\right|^{1+\al}}dy\right\|_{L^2(x)} < \infty\right\}, \quad \al \in (0,1) \\
X^{k+\log}_{c} & = \left\{f(x) \in \mathcal{C}_{w}(c), f \in X^{k}_{c}, f(x) = \sum_{j=1}^{\infty}a_{j} \cos(jx), \sum_{\pm} \left\|\int_{\mathbb{T}} \frac{\pa^{k}f(x\pm ic -y)-\pa^{k}f(x \pm ic)}{\left|\sin\left(\frac{y}{2}\right)\right|}dy\right\|_{L^2(x)} < \infty\right\} \\
X^{k+\log,m}_{c} & = \left\{f(x) \in \mathcal{C}_{w}(c), f \in X^{k}_{c}, f(x) = \sum_{j=1}^{\infty}a_{jm} \cos(jmx), \sum_{\pm} \left\|\int_{\mathbb{T}} \frac{\pa^{k}f(x \pm ic -y)-\pa^{k}f(x \pm ic)}{\left|\sin\left(\frac{y}{2}\right)\right|}dy\right\|_{L^2(x)} < \infty\right\}.
\end{align*}

The norm is given in the last two cases by the sum of the $X^{k}_{c}$-norm and the additional finite integral in the definition. 


\subsection{Theorems and outline of the proofs}

 The paper is organized as follows. In Section \ref{sectionproof}, we prove the following theorem:

\begin{theorem}
\label{teoremaestacionarias}
 Let $k\geq 3, m \in \mathbb{N}, m \geq 2, 0 < \alpha < 2$ and let $0 < b_{m}^{*} < 1$ be defined in Proposition \ref{propbstar}.
Then, there exists a family of $m$-fold stationary solutions $(b,R,r)$ and a $c > 0$, where $(R(x),r(x)) \in X^{k+1,m}_{c} \times X^{k+1,m}_{c}$ (for $\al < 1$), $(R(x),r(x)) \in X^{k+1+\log,m}_{c} \times X^{k+1+\log,m}_{c}$ (for $\al = 1$) or $(R(x),r(x)) \in X^{k+\al,m}_{c} \times X^{k+\al,m}_{c}$ (for $\al > 1$) of the equation \eqref{funcionalestacionarias} with $0 < \al < 2$ that bifurcate from the annulus of radii $1$ and $b_{m}^{*}$.
\end{theorem}

The proof will be carried out by means of a combination of a Crandall-Rabinowitz theorem and a priori estimates. Finally, in the Appendices we will include useful formulas and identities involving the special functions that appear throughout the proofs.

\section{Checking the hypotheses}
\label{sectionproof}

 The proof will be divided into 6 steps. These steps correspond to check the hypotheses of the Crandall-Rabinowitz theorem \cite{Crandall-Rabinowitz:bifurcation-simple-eigenvalues} for
 
 \begin{align*}
 F(b,R,r) = (F^{1}(b,R,r),F^{2}(b,R,r)),
 \end{align*}
 
 with
 
 \begin{align}
F^{1}(b,R,r) & = T_{1}(1+R) + T_{2}(b+r,1+R) \nonumber \\
F^{2}(b,R,r) & = -T_{2}(1+R,b+r) - T_{1}(b+r),
\end{align}

and

\begin{align*}
T_{1}(u) &= c_{\al} \int_{0}^{2\pi} \frac{\cos(y)(u'(x-y)u(x)-u(x-y)u'(x))}{\left(u(x)^2+u(x-y)^2-2u(x)u(x-y)\cos(y)\right)^{\al/2}}dy \\
& +c_{\al} \int_{0}^{2\pi} \frac{\sin(y)(u(x)u(x-y)+u'(x)u'(x-y)}{\left(u(x)^2+u(x-y)^2-2u(x)u(x-y)\cos(y)\right)^{\al/2}}dy \\
&= c_{\al} \int_{0}^{2\pi} \frac{\cos(y)(u'(x-y)u(x)-u(x-y)u'(x))}{\left(2-2\cos(y)\right)^{\al/2}}\left(\frac{2-2\cos(y)}{u(x)^2+u(x-y)^2-2u(x)u(x-y)\cos(y)}\right)^{\al/2}dy \\
& +c_{\al} \int_{0}^{2\pi} \frac{\sin(y)(u(x)u(x-y)+u'(x)u'(x-y)}{\left(2-2\cos(y)\right)^{\al/2}}\left(\frac{2-2\cos(y)}{u(x)^2+u(x-y)^2-2u(x)u(x-y)\cos(y)}\right)^{\al/2}dy \\
T_{2}(p,q) & = c_{\al} \int_{0}^{2\pi} \frac{\cos(y)(p(x-y)q'(x)-p'(x-y)q(x))}{\left(q(x)^2+p(x-y)^2-2q(x)p(x-y)\cos(y)\right)^{\al/2}}dy \\
& -c_{\al} \int_{0}^{2\pi} \frac{\sin(y)(p(x-y)q(x)+p'(x-y)q'(x))}{\left(q(x)^2+p(x-y)^2-2q(x)p(x-y)\cos(y)\right)^{\al/2}}dy \\
\end{align*}

The hypotheses are the following:

\begin{enumerate}

\item The functional $F$ satisfies $$F(b,R,r)\,:\, (0,1)\times \{V^{\ep}\}\mapsto Y^{k-1}_{c} \times Y^{k-1}_{c},$$ where $V^{\ep}$ is the open neighbourhood of 0

\begin{align*}
V^{\ep}=\left\{ 
\begin{array}{cc}
(f,g)\in X^{k}_{c} \times X^{k}_{c}\,:\, ||f||_{X^{k}_{c}}+||g||_{X^{k}_{c}}<\ep & \text{ if } \al < 1 \\
(f,g)\in X^{k+\log}_{c} \times X^{k+\log}_{c}\,:\, ||f||_{X^{k+\log}_{c}}+||g||_{X^{k+\log}_{c}}<\ep & \text{ if } \al = 1 \\
(f,g)\in X^{k+\al-1}_{c} \times X^{k+\al-1}_{c}\,:\, ||f||_{X^{k+\al-1}_{c}}+||g||_{X^{k+\al-1}_{c}}<\ep & \text{ if } \al > 1
\end{array}
\right.,
\end{align*}
for all $0<\ep<\ep_{0}(m)$ and $k\geq 3$.

\item $F(b,0,0) = 0$ for every $0 < b < 1$.
\item The partial derivatives $F_{r}$, $F_{R}$, $F_{bR}$ and $F_{br}$ exist and are continuous.
\item Ker($\mathcal{F}$) and $Y^{k-1}_{c}$/Range($\mathcal{F}$) are one-dimensional, where $\mathcal{F}$ is the linearized operator around $r = R = 0$ at $b = b^{*}_{m}$ (see Proposition \ref{propbstar} for a definition of $b_{m}^{*}$).
\item $\pa_{b} DF(b_{m}^{*},0,0)[h_0] \not \in$ Range($\mathcal{F}$), where Ker$(\mathcal{F}) = \langle h_0 \rangle$.
\item Step 1 can be applied to the spaces:

\begin{align*}
\left\{
\begin{array}{cc}
X^{k,m}_{c} \times X^{k,m}_{c} & \text{ if } \al < 1 \\
X^{k+\log,m}_{c} \times X^{k+\log,m}_{c} & \text{ if } \al = 1 \\
X^{k+\al-1,m}_{c} \times X^{k+\al-1,m}_{c} & \text{ if } \al > 1
\end{array}
\right.
\end{align*}

and $Y^{k-1,m}_{c} \times Y^{k-1,m}_{c}$ instead of

\begin{align*}
\left\{
\begin{array}{cc}
X^{k}_{c} \times X^{k}_{c} & \text{ if } \al < 1 \\
X^{k+\log}_{c} \times X^{k+\log}_{c} & \text{ if } \al = 1 \\
X^{k+\al-1}_{c} \times X^{k+\al-1}_{c} & \text{ if } \al > 1
\end{array}
\right.
\end{align*}

and $Y^{k-1}_{c} \times Y^{k-1}_{c}$ respectively.

\end{enumerate}

\begin{rem}
For the choices of $u$ that will appear in the Theorem (of the form constant $+ O(\varepsilon)$), the function inside the parenthesis in $T_1(u)$ is uniformly bounded from below in $y$ for every $x$
by a strictly positive constant. Then we can analytically extend the integrand in $x$ to the strip $|\Im(z)| \leq c$ for a small enough $c$.
\end{rem}

\subsection{Step 1}

The regularity step of the functional $F$ was already shown in \cite{Renault:relative-equlibria-holes-sqg} for $\al = 1$, in \cite{delaHoz-Hassainia-Hmidi:doubly-connected-vstates-gsqg} for $\al < 1$ and can be easily adapted from the proof of \cite{Castro-Cordoba-GomezSerrano:analytic-vstates-ellipses} for $\al > 1$.

\subsection{Step 2}

This is trivial since $T_1(1)$, $T_1(b)$, $T_2(1,b)$ and $T_2(b,1)$ consist of integrands which are either zero or odd (and therefore have integral zero).

\subsection{Step 3}

We need to prove the existence and the continuity of the Gateaux derivatives $\pa_{R} F(b,R,r)$, $\pa_{r} F(b,R,r)$, $\pa_{bR} F(b,R,r)$ and $\pa_{br} F(b,R,r)$. We have the following Lemma:

\begin{lemma}
\label{lemagatoderivada}
 For all $(R,r)\in V^r$ and for all $(H,h)\in X$, where $X = (X^{k}_{c} \times X^{k}_{c}), (X^{k+\log}_{c} \times X^{k+\log}_{c})$ or $(X^{k+\al-1}_{c} \times X^{k+\al-1}_{c})$ depending on $\alpha$, such that $||(h,H)||_{X}=1$ we have that:
 
\begin{align}
D_{R}F^{1}(b,R,r)[H] & = c_{\al} \int_{0}^{2\pi} \frac{\cos(y)(H'(x-y)(1+R(x))+R'(x-y)H(x)-H(x-y)R'(x)-(1+R(x-y))H'(x))}{\left(2-2\cos(y)\right)^{\al/2}} \nonumber \\
& \times \left(\frac{2-2\cos(y)}{(1+R(x))^2+(1+R(x-y))^2-2(1+R(x))(1+R(x-y))\cos(y)}\right)^{\al/2}dy \nonumber \\
& - \left(\frac{\al}{2}\right)c_{\al} \int_{0}^{2\pi} \frac{\cos(y)(R'(x-y)(1+R(x))-(1+R(x-y))R'(x))}{\left(2-2\cos(y)\right)^{\al/2}} \nonumber \\
& \times 2\left(\frac{(1+R(x))H(x)+(1+R(x-y))H(x-y)-((1+R(x))H(x-y)+(1+R(x-y))H(x))\cos(y)}{2-2\cos(y)}\right) \nonumber \\
& \times \left(\frac{2-2\cos(y)}{(1+R(x))^2+(1+R(x-y))^2-2(1+R(x))(1+R(x-y))\cos(y)}\right)^{\al/2+1}dy \nonumber \\
& +c_{\al} \int_{0}^{2\pi} \frac{\sin(y)((1+R(x))H(x-y) + H(x)(1+R(x-y)) + R'(x)H'(x-y)+H'(x)R'(x-y)}{\left(2-2\cos(y)\right)^{\al/2}}\nonumber \\
& \times \left(\frac{2-2\cos(y)}{(1+R(x))^2+(1+R(x-y))^2-2(1+R(x))(1+R(x-y))\cos(y)}\right)^{\al/2}dy \nonumber \\
& - \left(\frac{\al}{2}\right)c_{\al} \int_{0}^{2\pi} \frac{\sin(y)((1+R(x-y))(1+R(x))+R'(x-y)R'(x))}{\left(2-2\cos(y)\right)^{\al/2}} \nonumber \\
& \times 2\left(\frac{(1+R(x))H(x)+(1+R(x-y))H(x-y)-((1+R(x))H(x-y)+(1+R(x-y))H(x))\cos(y)}{2-2\cos(y)}\right) \nonumber \\
& \times \left(\frac{2-2\cos(y)}{(1+R(x))^2+(1+R(x-y))^2-2(1+R(x))(1+R(x-y))\cos(y)}\right)^{\al/2+1}dy \nonumber \\
& +c_{\al} \int_{0}^{2\pi} \frac{\cos(y)((b+r(x-y))H'(x)-r'(x-y)H(x))}{\left((1+R(x))^2+(b+r(x-y))^2-2(1+R(x))(b+r(x-y))\cos(y)\right)^{\al/2}}dy \nonumber \\
& - \left(\frac{\al}{2}\right)c_{\al} \int_{0}^{2\pi} \left(\cos(y)((b+r(x-y))R'(x)-r'(x-y)(1+R(x)))\right)\nonumber \\
& \times 2\left(\frac{(1+R(x))H(x)-(b+r(x-y))H(x)\cos(y)}{\left((1+R(x))^2+(b+r(x-y))^2-2(1+R(x))(b+r(x-y))\cos(y)\right)^{\al/2+1}}\right)dy \nonumber \\
& -c_{\al} \int_{0}^{2\pi} \frac{\sin(y)(H(x)(b+r(x-y)) + H'(x)r'(x-y))}{\left((1+R(x))^2+(b+r(x-y))^2-2(1+R(x))(b+r(x-y))\cos(y)\right)^{\al/2}}dy\nonumber \\
& + \left(\frac{\al}{2}\right)c_{\al} \int_{0}^{2\pi} \left(\sin(y)((b+r(x-y))(1+R(x))+r'(x-y)R'(x))\right) \nonumber \\
& \times 2\left(\frac{(1+R(x))H(x)-(b+r(x-y))H(x)\cos(y)}{\left((1+R(x))^2+(b+r(x-y))^2-2(1+R(x))(b+r(x-y))\cos(y)\right)^{\al/2+1}}\right)dy \nonumber \\
D_{r}F^{1}(b,R,r)[h] & = 
 c_{\al} \int_{0}^{2\pi} \frac{\cos(y)(h(x-y)R'(x)-h'(x-y)(1+R(x)))}{\left((1+R(x))^2+(b+r(x-y))^2-2(1+R(x))(b+r(x-y))\cos(y)\right)^{\al/2}}dy \nonumber \\
& - \left(\frac{\al}{2}\right)c_{\al} \int_{0}^{2\pi} \left(\cos(y)((b+r(x-y))R'(x)-r'(x-y)(1+R(x)))\right)\nonumber \\
& \times 2\left(\frac{(b+r(x))h(x)-(1+R(x))h(x-y)\cos(y)}{\left((1+R(x))^2+(b+r(x-y))^2-2(1+R(x))(b+r(x-y))\cos(y)\right)^{\al/2+1}}\right)dy \nonumber \\
& -c_{\al} \int_{0}^{2\pi} \frac{\sin(y)((1+R(x))h(x-y) + R'(x)h'(x-y))}{\left((1+R(x))^2+(b+r(x-y))^2-2(1+R(x))(b+r(x-y))\cos(y)\right)^{\al/2}}dy\nonumber \\
& + \left(\frac{\al}{2}\right)c_{\al} \int_{0}^{2\pi} \left(\sin(y)((b+r(x-y))(1+R(x))+r'(x-y)R'(x))\right) \nonumber \\
& \times 2\left(\frac{(b+r(x))h(x)-(1+R(x))h(x-y)\cos(y)}{\left((1+R(x))^2+(b+r(x-y))^2-2(1+R(x))(b+r(x-y))\cos(y)\right)^{\al/2+1}}\right)dy \nonumber \\
D_{R}F^{2}(b,R,r)[H] & = 
 -c_{\al} \int_{0}^{2\pi} \frac{\cos(y)(H(x-y)r'(x)-H'(x-y)(b+r(x)))}{\left((b+r(x))^2+(1+R(x-y))^2-2(b+r(x))(1+R(x-y))\cos(y)\right)^{\al/2}}dy \nonumber \\
& + \left(\frac{\al}{2}\right)c_{\al} \int_{0}^{2\pi} \left(\cos(y)((1+R(x-y))r'(x)-R'(x-y)(b+r(x)))\right)\nonumber \\
& \times 2\left(\frac{(1+R(x-y))H(x-y)-(b+r(x))H(x-y)\cos(y)}{\left((b+r(x))^2+(1+R(x-y))^2-2(b+r(x))(1+R(x-y))\cos(y)\right)^{\al/2+1}}\right)dy \nonumber \\
& +c_{\al} \int_{0}^{2\pi} \frac{\sin(y)((b+r(x))H(x-y) + r'(x)H'(x-y))}{\left((b+r(x))^2+(1+R(x-y))^2-2(b+r(x))(1+R(x-y))\cos(y)\right)^{\al/2}}dy\nonumber \\
& - \left(\frac{\al}{2}\right)c_{\al} \int_{0}^{2\pi} \left(\sin(y)((1+R(x-y))(b+r(x))+R'(x-y)r'(x))\right) \nonumber \\
& \times 2\left(\frac{(1+R(x-y))H(x-y)-(b+r(x))H(x-y)\cos(y)}{\left((b+r(x))^2+(1+R(x-y))^2-2(b+r(x))(1+R(x-y))\cos(y)\right)^{\al/2+1}}\right)dy \nonumber \\
D_{r}F^{2}(b,R,r)[h] & = 
-c_{\al} \int_{0}^{2\pi} \frac{\cos(y)(h'(x-y)(b+r(x))+r'(x-y)h(x)-h(x-y)r'(x)-(b+r(x-y))h'(x))}{\left(2-2\cos(y)\right)^{\al/2}} \nonumber \\
& \times \left(\frac{2-2\cos(y)}{(b+r(x))^2+(b+r(x-y))^2-2(b+r(x))(b+r(x-y))\cos(y)}\right)^{\al/2}dy \nonumber \\
& + \left(\frac{\al}{2}\right)c_{\al} \int_{0}^{2\pi} \frac{\cos(y)(r'(x-y)(b+r(x))-(b+r(x-y))r'(x))}{\left(2-2\cos(y)\right)^{\al/2}} \nonumber \\
& \times 2\left(\frac{(b+r(x))h(x)+(b+r(x-y))h(x-y)-((b+r(x))h(x-y)+(b+r(x-y))h(x))\cos(y)}{2-2\cos(y)}\right) \nonumber \\
& \times \left(\frac{2-2\cos(y)}{(b+r(x))^2+(b+r(x-y))^2-2(b+r(x))(b+r(x-y))\cos(y)}\right)^{\al/2+1}dy \nonumber \\
& -c_{\al} \int_{0}^{2\pi} \frac{\sin(y)((b+r(x))h(x-y) + h(x)(b+r(x-y)) + r'(x)h'(x-y)+h'(x)r'(x-y)}{\left(2-2\cos(y)\right)^{\al/2}}\nonumber \\
& \times \left(\frac{2-2\cos(y)}{(b+r(x))^2+(b+r(x-y))^2-2(b+r(x))(b+r(x-y))\cos(y)}\right)^{\al/2}dy \nonumber \\
& + \left(\frac{\al}{2}\right)c_{\al} \int_{0}^{2\pi} \frac{\sin(y)((b+r(x-y))(b+r(x))+r'(x-y)r'(x))}{\left(2-2\cos(y)\right)^{\al/2}} \nonumber \\
& \times 2\left(\frac{(b+r(x))h(x)+(b+r(x-y))h(x-y)-((b+r(x))h(x-y)+(b+r(x-y))h(x))\cos(y)}{2-2\cos(y)}\right) \nonumber \\
& \times \left(\frac{2-2\cos(y)}{(b+r(x))^2+(b+r(x-y))^2-2(b+r(x))(b+r(x-y))\cos(y)}\right)^{\al/2+1}dy \nonumber \\
& - c_{\al} \int_{0}^{2\pi} \frac{\cos(y)((1+R(x-y))h'(x)-R'(x-y)h(x))}{\left((b+r(x))^2+(1+R(x-y))^2-2(b+r(x))(1+R(x-y))\cos(y)\right)^{\al/2}}dy \nonumber \\ 
& + \left(\frac{\al}{2}\right)c_{\al} \int_{0}^{2\pi} \left(\cos(y)((1+R(x-y))r'(x)-R'(x-y)(b+r(x)))\right)\nonumber \\
& \times 2\left(\frac{(b+r(x))h(x)-(1+R(x-y))h(x)\cos(y)}{\left((b+r(x))^2+(1+R(x-y))^2-2(b+r(x))(1+R(x-y))\cos(y)\right)^{\al/2+1}}\right)dy \nonumber \\
& +c_{\al} \int_{0}^{2\pi} \frac{\sin(y)(h(x)(1+R(x-y)) + h'(x)R'(x-y))}{\left((b+r(x))^2+(1+R(x-y))^2-2(b+r(x))(1+R(x-y))\cos(y)\right)^{\al/2}}dy\nonumber \\
& - \left(\frac{\al}{2}\right)c_{\al} \int_{0}^{2\pi} \left(\sin(y)((1+R(x-y))(b+r(x))+R'(x-y)r'(x))\right) \nonumber  \\
& \times 2\left(\frac{(b+r(x))h(x)-(1+R(x-y))h(x)\cos(y)}{\left((b+r(x))^2+(1+R(x-y))^2-2(b+r(x))(1+R(x-y))\cos(y)\right)^{\al/2+1}}\right)dy \label{derivadalarga}
\end{align} 

Moreover, these functions are continuous in $(R,r)$.
\end{lemma}

\begin{proof}
Straightforward computation.

The continuity of $\pa_{r}F(b,R,r)$ and $\pa_{R} F(b,R,r)$ was done in \cite{Renault:relative-equlibria-holes-sqg} for $\al = 1$, and in \cite{delaHoz-Hassainia-Hmidi:doubly-connected-vstates-gsqg} for $\al < 1$ for H\"older-based spaces but it can easily be extended to the case $\al > 1$ and Sobolev-based spaces using the same techniques.

We explain now how to deal with derivatives with respect to $b$. The only problematic terms are the ones that contain a factor such as the one below in brackets (the first term in $\pa_{r}F^{2}(b,R,r)[h]$):

\begin{align*}
A(b,x) & =-c_{\al} \int_{0}^{2\pi} \frac{\cos(y)(h'(x-y)(b+r(x))+r'(x-y)h(x)-h(x-y)r'(x)-(b+r(x-y))h'(x))}{\left(2-2\cos(y)\right)^{\al/2}} \nonumber \\
& \times \left(\frac{2-2\cos(y)}{(b+r(x))^2+(b+r(x-y))^2-2(b+r(x))(b+r(x-y))\cos(y)}\right)^{\al/2}dy
\end{align*}

Taking a derivative in $b$:

\begin{align*}
\pa_{b} A(b,x) & =-c_{\al} \int_{0}^{2\pi} \frac{\cos(y)(h'(x-y)-h'(x))}{\left(2-2\cos(y)\right)^{\al/2}} \left(\frac{2-2\cos(y)}{(b+r(x))^2+(b+r(x-y))^2-2(b+r(x))(b+r(x-y))\cos(y)}\right)^{\al/2}dy \\
& +c_{\al}\left(\frac{\al}{2}\right) \int_{0}^{2\pi} \frac{\cos(y)(h'(x-y)(b+r(x))+r'(x-y)h(x)-h(x-y)r'(x)-(b+r(x-y))h'(x))}{\left(2-2\cos(y)\right)^{\al/2}} \nonumber \\
& \times \left(\frac{2-2\cos(y)}{(b+r(x))^2+(b+r(x-y))^2-2(b+r(x))(b+r(x-y))\cos(y)}\right)^{\al/2+1} \\
& \times \left(\frac{(2-2\cos(y))^{2}(b+r(x)+b+r(x-y))}{((b+r(x))^2+(b+r(x-y))^2-2(b+r(x))(b+r(x-y))\cos(y))^{2}}\right)
dy \\
& = A_{1}(b,x) + A_{2}(b,x),
\end{align*}

and both terms can be shown to be bounded and continuous as in the cases of $\pa_{r}F(b,R,r)$ or $\pa_{R}F(b,R,r)$.
\end{proof}

\subsection{Step 4}

\subsubsection{Calculation of $\mathcal{F}$}

Before proving Step 4, we compute the linearization of $F$ around $(0,0)$ in the direction $(h(x),H(x))$. Note that this is also obtainable from the computation in \cite{delaHoz-Hassainia-Hmidi:doubly-connected-vstates-gsqg} by setting $\Omega = 0$. 

\begin{prop}
Let $\displaystyle h(x) = \sum_{n}a_n \cos(nx), H(x) = \sum_{n}A_n\cos(nx)$, then we have that:

\begin{align*}
DF(b,0,0)[H,h] = \left(\begin{array}{c}U(x) \\ u(x) \end{array}\right),
\end{align*}

where

\begin{align*}
u(x) = \sum_{n} c_n \sin(nx), \quad U(x) = \sum_{n}U_n \sin(nx),
\end{align*}

and the coefficients satisfy, for any $n$:

\begin{align*}
(-n) M_n^{\al}(b)
\left(\begin{array}{c}A_n \\ a_n \end{array}\right)
= 
(-n) \left(
\begin{array}{cc}
 -\Theta_n + b^{2} \Lambda_1(b) & - b^{2}\Lambda_{n}(b) \\
b \Lambda_n(b) & b^{1-\al}\Theta_n - b \Lambda_1(b)
\end{array}
\right)
\left(\begin{array}{c}A_n \\ a_n \end{array}\right) = \left(\begin{array}{c}U_n \\ u_n \end{array}\right)
\end{align*}

with

\begin{align}
 \Lambda_n(b) & \equiv \frac{1}{b} \int_{0}^{\infty} \frac{1}{t^{1-\al}}J_n(bt)J_n(t) dt \nonumber \\
& = \frac{\Gamma\left(\frac{\al}{2}\right)}{\Gamma\left(1-\frac{\al}{2}\right)2^{1-\al}} \frac{\left(\frac{\al}{2}\right)_{n}}{n!}b^{n-1} F\left(\frac{\al}{2}, n+\frac{\al}{2},n+1,b^{2}\right), \nonumber \\
& = \frac{b^{n-1}}{2^{1-\al}\Gamma\left(1-\frac{\al}{2}\right)^{2}} \int_{0}^{1} x^{n-1+\frac{\al}{2}}(1-x)^{-\frac{\al}{2}}(1-b^{2}x)^{-\frac{\al}{2}}dx. \label{defilambdan} \\
\Theta_{n} & \equiv \Lambda_{1}(1) - \Lambda_{n}(1) \nonumber
\end{align}

\end{prop}

\begin{proof}

We first start by setting $r = R = 0$ in \eqref{derivadalarga}, yielding:

\begin{align*}
D_{R}F^{1}(b,0,0)[H] & = c_{\al} \int_{0}^{2\pi} \frac{\cos(y)(H'(x-y)-H'(x))}{\left(2-2\cos(y)\right)^{\al/2}}dy +c_{\al} \int_{0}^{2\pi} \frac{\sin(y)(H(x-y) + H(x))}{\left(2-2\cos(y)\right)^{\al/2}} dy\\
& - \left(\frac{\al}{2}\right)c_{\al} \int_{0}^{2\pi} \frac{\sin(y)(H(x)+H(x-y))}{\left(2-2\cos(y)\right)^{\al/2}}dy
 +c_{\al} \int_{0}^{2\pi} \frac{\cos(y)(bH'(x))}{\left(1+b^2 - 2b\cos(y)\right)^{\al/2}}dy \\
& -c_{\al} \int_{0}^{2\pi} \frac{\sin(y)(H(x)b)}{\left(1+b^2-2b\cos(y)\right)^{\al/2}}dy
+ 2\left(\frac{\al}{2}\right)c_{\al} \int_{0}^{2\pi} \left(\sin(y)(b)\right)\left(\frac{H(x)(1-b\cos(y))}{\left(1+b^2-2b\cos(y)\right)^{\al/2+1}}\right)dy \\
& = c_{\al} \int_{0}^{2\pi} \frac{\cos(y)(H'(x-y)-H'(x))}{\left(2-2\cos(y)\right)^{\al/2}}dy +c_{\al} \int_{0}^{2\pi} \frac{\sin(y)(H(x-y) + H(x))}{\left(2-2\cos(y)\right)^{\al/2}} dy\\
& - \left(\frac{\al}{2}\right)c_{\al} \int_{0}^{2\pi} \frac{\sin(y)(H(x)+H(x-y))}{\left(2-2\cos(y)\right)^{\al/2}}dy
 +c_{\al} \int_{0}^{2\pi} \frac{\cos(y)(bH'(x))}{\left(1+b^2 - 2b\cos(y)\right)^{\al/2}}dy \\
D_{r}F^{1}(b,0,0)[h] & = 
 -c_{\al} \int_{0}^{2\pi} \frac{\cos(y)(h'(x-y))}{\left(1+b^2-2b\cos(y)\right)^{\al/2}}dy
-c_{\al} \int_{0}^{2\pi} \frac{\sin(y)(h(x-y))}{\left(1+b^2-2b\cos(y)\right)^{\al/2}}dy\\
& + 2\left(\frac{\al}{2}\right)c_{\al} \int_{0}^{2\pi} \left(\sin(y)b\right)\left(\frac{h(x-y)(b-\cos(y))}{\left(1+b^2-2b\cos(y)\right)^{\al/2+1}}\right)dy \\
D_{R}F^{2}(b,0,0)[H] & = 
 c_{\al} \int_{0}^{2\pi} \frac{\cos(y)(H'(x-y)b)}{\left(1+b^2-2b\cos(y)\right)^{\al/2}}dy +c_{\al} \int_{0}^{2\pi} \frac{\sin(y)(bH(x-y))}{\left(1+b^2-2b\cos(y)\right)^{\al/2}}dy\\
& - 2\left(\frac{\al}{2}\right)c_{\al} \int_{0}^{2\pi} \left(\sin(y)(b)\right)\left(\frac{H(x-y)(1-b\cos(y))}{\left(1+b^2-2b\cos(y)\right)^{\al/2+1}}\right)dy \\
D_{r}F^{2}(b,0,0)[h] & = 
-c_{\al} \int_{0}^{2\pi} \frac{b\cos(y)(h'(x-y)-h'(x))}{\left(2-2\cos(y)\right)^{\al/2}} \frac{1}{b^{\al}} dy 
-c_{\al} \int_{0}^{2\pi} \frac{b\sin(y)(h(x-y) + h(x))}{\left(2-2\cos(y)\right)^{\al/2}} \frac{1}{b^{\al}}dy\\
& + \left(\frac{\al}{2}\right)c_{\al} \int_{0}^{2\pi} \frac{\sin(y)(b^2)b(h(x)+h(x-y))}{\left(2-2\cos(y)\right)^{\al/2}} \frac{1}{b^{\al+2}}dy - c_{\al} \int_{0}^{2\pi} \frac{\cos(y)(h'(x))}{\left(1+b^2-2b\cos(y)\right)^{\al/2}}dy \\ 
& +c_{\al} \int_{0}^{2\pi} \frac{\sin(y)(h(x))}{\left(1+b^2-2b\cos(y)\right)^{\al/2}}dy - 2\left(\frac{\al}{2}\right)c_{\al} \int_{0}^{2\pi} \left(\sin(y)b\right) \left(\frac{h(x)(b-\cos(y))}{\left(1+b^2-2b\cos(y)\right)^{\al/2+1}}\right)dy \\
& = -c_{\al} \int_{0}^{2\pi} \frac{b\cos(y)(h'(x-y)-h'(x))}{\left(2-2\cos(y)\right)^{\al/2}} \frac{1}{b^{\al}} dy 
-c_{\al} \int_{0}^{2\pi} \frac{b\sin(y)(h(x-y) + h(x))}{\left(2-2\cos(y)\right)^{\al/2}} \frac{1}{b^{\al}}dy\\
& + \left(\frac{\al}{2}\right)c_{\al} \int_{0}^{2\pi} \frac{\sin(y)(b^2)b(h(x)+h(x-y))}{\left(2-2\cos(y)\right)^{\al/2}} \frac{1}{b^{\al+2}}dy - c_{\al} \int_{0}^{2\pi} \frac{\cos(y)(h'(x))}{\left(1+b^2-2b\cos(y)\right)^{\al/2}}dy \\ 
\end{align*} 

We now integrate by parts and obtain:

\begin{align*}
D_{r}F^{1}(b,0,0)[h] & = 2b^2\left(\frac{\al}{2}\right)c_{\al} \int_{0}^{2\pi} \left(\frac{\sin(y)h(x-y)}{\left(1+b^2-2b\cos(y)\right)^{\al/2+1}}\right)dy \\
D_{R}F^{2}(b,0,0)[H] & = - 2b\left(\frac{\al}{2}\right)c_{\al} \int_{0}^{2\pi}\left(\frac{\sin(y)H(x-y)}{\left(1+b^2-2b\cos(y)\right)^{\al/2+1}}\right)dy \\
\end{align*}

By linearity, it suffices to do the calculations when $H(x) = A_n\cos(nx), h(x) = a_n \cos(nx)$. In that case:

\begin{align*}
D_{r}F^{1}(b,0,0)[h] & = 2a_nb^2\left(\frac{\al}{2}\right)c_{\al} \sin(nx) \int_{0}^{2\pi} \left(\frac{\sin(y)\sin(ny)}{\left(1+b^2-2b\cos(y)\right)^{\al/2+1}}\right)dy \\
D_{R}F^{2}(b,0,0)[H] & = - 2A_nb\left(\frac{\al}{2}\right)c_{\al}\sin(nx) \int_{0}^{2\pi}\left(\frac{\sin(y)\sin(ny)}{\left(1+b^2-2b\cos(y)\right)^{\al/2+1}}\right)dy \\
\end{align*}

Using Lemma \ref{lemmaA2}, this shows the off-diagonal entries of $M_n^{\al}(b)$.

We finally move on to the terms in $D_{R}F^{1}$ and $D_{r}F^{2}$. The sums of each of the first three terms were calculated before in \cite{Castro-Cordoba-GomezSerrano:existence-regularity-vstates-gsqg,Hassainia-Hmidi:v-states-generalized-sqg} and equal $n \Theta_{n}$ and $-nb^{1-\al}\Theta_{n}$ respectively. The fourth one can be calculated using Lemma \ref{lemaA1} with $m = 1$. This completes the proof of the Proposition.

\end{proof}

\subsubsection{One dimensionality of the Kernel of the linear operator.}

We will start computing a nontrivial element of the kernel of $DF[b,0,0]\left(\begin{array}{c} H \\ h \end{array}\right)$, where

\begin{align*}
H(x) = \sum_{n=1}^{\infty} A_n \cos(nx), \quad h(x) = \sum_{n=1}^{\infty} a_n \cos(nx).
\end{align*}

We have that

\begin{align*}
 DF[b,0,0]\left(\begin{array}{c} H \\ h \end{array}\right) = \sum_{n=1}^{\infty} (-n) M_{n}^{\al}(b) \left(\begin{array}{c} A_n \\ a_n\end{array}\right) \sin(nx),
\end{align*}

where $M_{n}^{\al}(b)$ and $\Lambda_{n}(b)$ were defined in \eqref{defilambdan}.

\begin{lemma}\label{lemajcreciente}
 Let $\al \in (0,2)$ and $n \geq 2$. Then:
\begin{align*}
 j(b) = \left(\frac{\Lambda_{n}(b)}{\Lambda_{1}(b)}\right)^{2}
\end{align*}

is a positive, increasing function of $b$.

\end{lemma}

\begin{proof}
 Since $\frac{\Lambda_{n}(b)}{\Lambda_{1}(b)}$ is positive by Lemma \ref{lemmaA2}, it is enough to show that it is increasing. To do so, we will show that $\frac{\Lambda_{n}'(b)}{\Lambda_{n}(b)} - \frac{\Lambda_{1}'(b)}{\Lambda_{1}(b)} > 0$. Using the integral representation of $\Lambda_{n}(b)$:

\begin{align*}
 \Lambda_{n}(b) & = \frac{b^{n-1}}{2^{1-\al}\Gamma\left(1-\frac{\al}{2}\right)^{2}} \int_{0}^{1} x^{n-1+\frac{\al}{2}}(1-x)^{-\frac{\al}{2}}(1-b^{2}x)^{-\frac{\al}{2}}dx
\end{align*}

one obtains that

\begin{align*}
 \frac{\Lambda_{n}'(b)}{\Lambda_{n}(b)} & = \frac{n-1}{b} + \al b\frac{\int_{0}^{1} x^{n+\frac{\al}{2}}(1-x)^{-\frac{\al}{2}}(1-b^{2}x)^{-\frac{\al}{2}-1}dx}{\int_{0}^{1} x^{n-1+\frac{\al}{2}}(1-x)^{-\frac{\al}{2}}(1-b^{2}x)^{-\frac{\al}{2}}dx} \\
 \frac{\Lambda_{1}'(b)}{\Lambda_{1}(b)} & =  \al b\frac{\int_{0}^{1} x^{1+\frac{\al}{2}}(1-x)^{-\frac{\al}{2}}(1-b^{2}x)^{-\frac{\al}{2}-1}dx}{\int_{0}^{1} x^{\frac{\al}{2}}(1-x)^{-\frac{\al}{2}}(1-b^{2}x)^{-\frac{\al}{2}}dx} \\
\end{align*}

Thus, $\frac{\Lambda_{n}'(b)}{\Lambda_{n}(b)} - \frac{\Lambda_{1}'(b)}{\Lambda_{1}(b)} > 0$ iff

\begin{align*}
 \frac{\int_{0}^{1} x^{n+\frac{\al}{2}}(1-x)^{-\frac{\al}{2}}(1-b^{2}x)^{-\frac{\al}{2}-1}dx}{\int_{0}^{1} x^{n-1+\frac{\al}{2}}(1-x)^{-\frac{\al}{2}}(1-b^{2}x)^{-\frac{\al}{2}}dx}  > \frac{\int_{0}^{1} x^{1+\frac{\al}{2}}(1-x)^{-\frac{\al}{2}}(1-b^{2}x)^{-\frac{\al}{2}-1}dx}{\int_{0}^{1} x^{\frac{\al}{2}}(1-x)^{-\frac{\al}{2}}(1-b^{2}x)^{-\frac{\al}{2}}dx} \\
\Leftrightarrow
 \int_{0}^{1} \int_{0}^{1} x^{n+\frac{\al}{2}}(1-x)^{-\frac{\al}{2}}(1-b^{2}x)^{-\frac{\al}{2}-1}y^{\frac{\al}{2}}(1-y)^{-\frac{\al}{2}}(1-b^{2}y)^{-\frac{\al}{2}}dxdy  \\
  > \int_{0}^{1} \int_{0}^{1} y^{1+\frac{\al}{2}}(1-y)^{-\frac{\al}{2}}(1-b^{2}y)^{-\frac{\al}{2}-1} x^{n-1+\frac{\al}{2}}(1-x)^{-\frac{\al}{2}}(1-b^{2}x)^{-\frac{\al}{2}}dy \\
\Leftrightarrow
\int_{0}^{1} \int_{0}^{1} (1-x)^{-\frac{\al}{2}}(1-y)^{-\frac{\al}{2}}x^{\frac{\al}{2}}y^{\frac{\al}{2}}(1-b^{2}x)^{-\frac{\al}{2}-1}(1-b^{2}y)^{-\frac{\al}{2}-1}(x^{n}(1-b^{2}y) - yx^{n-1}(1-b^{2}x))dxdy  > 0\\
\Leftrightarrow
\int_{0}^{1} \int_{0}^{1} (1-x)^{-\frac{\al}{2}}(1-y)^{-\frac{\al}{2}}x^{\frac{\al}{2}}y^{\frac{\al}{2}}(1-b^{2}x)^{-\frac{\al}{2}-1}(1-b^{2}y)^{-\frac{\al}{2}-1}x^{n-1}(x-y)dx dy  > 0\\
\Leftrightarrow
\frac12 \int_{0}^{1} \int_{0}^{1} (1-x)^{-\frac{\al}{2}}(1-y)^{-\frac{\al}{2}}x^{\frac{\al}{2}}y^{\frac{\al}{2}}(1-b^{2}x)^{-\frac{\al}{2}-1}(1-b^{2}y)^{-\frac{\al}{2}-1}(x^{n-1}-y^{n-1})(x-y)dx dy  > 0,
\end{align*}

which is true since the integrand is positive.

\end{proof}

We can prove the following proposition:

\begin{prop}\label{propbstar}
Let $\Delta_{m}^{\al}(b)$ be 

\begin{align*}
 \Delta_{m}^{\al}(b) = \text{det}(M_{m}^{\al}(b)) = (-\Theta_{m} + b^{2}\Lambda_1(b))(b^{1-\al}\Theta_{m} - b\Lambda_{1}(b)) + b^3\Lambda_{m}(b)^{2}
\end{align*}

Then, for any $\al \in (0,2)$ and for any $m \geq 2$, there exists a unique $b_{m}^{*}$ such that $\Delta_{m}^{\al}(b_{m}^{*}) = 0$. We also have that rk$(M_{m}^{\al}(b_{m}^{*})) = 1$ for that value of $b_{m}^{*}$.

Moreover, for fixed $\al \in (0,2)$, the sequence $b_{m}^{*}$ is increasing in $m$.
\end{prop}

\begin{proof}
 We first show the existence of $b_{m}^{*}$. Fix $\alpha$ and $m$. Expanding $\Delta_{m}^{\al}(b)$, we obtain:

\begin{align*}
 \Delta_{m}^{\al}(b) & = -b^{1-\al}\Theta_{m}^{2} + \Theta_{m}(b^{3-\al}\Lambda_1(b) + b \Lambda_1(b)) + b^{3}(\Lambda_{m}(b)^{2} - \Lambda_1(b)^{2}).
\end{align*}

If $b_m$ is a solution of $\Delta_m^{\al}(b_m) = 0$, then 

\begin{align}
\label{defiQ}
 \Theta_{m} & = \frac{1}{2b^{1-\al}}\left(\Lambda_1(b)(b + b^{3-\al}) \pm \sqrt{\Lambda_1(b)^{2}(b+b^{3-\al})^{2} - 4b^{4-\al}(\Lambda_1(b)^{2} - \Lambda_m(b)^{2})}\right) \equiv Q_{\pm}(b,m)
\end{align}

at $b = b_m$. We note that both $Q_{\pm}(b,m)$ are real since the discriminant is equal to $\Lambda_{1}(b)^{2}(b-b^{3-\al})^{2} + 4b^{4-\al} \Lambda_{m}(b)^{2} \geq 0$. This also implies $Q_{-}(b,m) \leq Q_{+}(b,m)$ for all $b,m$.

\begin{prop}

Let $m \geq 2$ and let $Q_{-}(b,m)$ be defined as in \eqref{defiQ}. We have that, for all $0 < b \leq 1$:

\begin{align*}
 Q_{-}(b,m) \leq \Theta_m,
\end{align*}

with equality only if $b = 1$. 

\end{prop}

\begin{proof}
We start with the following chain of inequalities:

\begin{align*}
 Q_{-}(b,m) & = \frac12 \frac{(b^{\al}+b^{2})^{2}\Lambda_1(b)^{2} - ((b^{\al}-b^{2})^{2}\Lambda_{1}(b)^{2} + 4b^{2+\al}\Lambda_{m}(b)^{2})}{\left(\Lambda_1(b)(b^{\al} + b^{2}) + \sqrt{\Lambda_1(b)^{2}(b^{\al}-b^{2})^{2} + 4b^{2+\al}\Lambda_m(b)^{2}}\right)} \\
& = \frac{2b^{2+\al}(\Lambda_{1}(b)^{2} - \Lambda_{m}(b)^{2})}{\left(\Lambda_1(b)(b^{\al} + b^{2}) + \sqrt{\Lambda_1(b)^{2}(b^{\al}-b^{2})^{2} + 4b^{2+\al}\Lambda_m(b)^{2}}\right)} \\
& = \frac{2b^{2}(\Lambda_{1}(b)^{2} - \Lambda_{m}(b)^{2})}{\left(\Lambda_1(b)(1 + b^{2-\al}) + \sqrt{\Lambda_1(b)^{2}(1-b^{2-\al})^{2} + 4b^{2-\al}\Lambda_m(b)^{2}}\right)} \\
& \leq \frac{2b^{2}(\Lambda_{1}(b)^{2} - \Lambda_{m}(b)^{2})}{\left(\Lambda_1(b)(1 + b^{2-\al}) + \sqrt{\Lambda_m(b)^{2}(1-b^{2-\al})^{2} + 4b^{2-\al}\Lambda_m(b)^{2}}\right)} \\
& = \frac{2b^{2}(\Lambda_{1}(b)^{2} - \Lambda_{m}(b)^{2})}{(\Lambda_1(b)+\Lambda_m(b))(1 + b^{2-\al})}
 = \frac{2b^{2}}{(1 + b^{2-\al})}(\Lambda_{1}(b) - \Lambda_{m}(b)) \leq  b^{1+\frac{\al}{2}}(\Lambda_{1}(b) - \Lambda_{m}(b)).
\end{align*}

We claim that

\begin{align*}
  b^{1+\frac{\al}{2}}(\Lambda_{1}(b) - \Lambda_{m}(b)) \leq (\Lambda_{1}(1) - \Lambda_{m}(1)) = \Theta_{m}.
\end{align*}

In order to prove it, we will show that the LHS is an increasing function of $b$. This is enough since both LHS and RHS agree at $b = 1$. Taking a derivative, we obtain:

\begin{align*}
 b^{\frac{\al}{2}}\left(\left(1+\frac{\al}{2}\right)(\Lambda_{1}(b) - \Lambda_{m}(b)) + b(\Lambda_{1}'(b) - \Lambda_{m}'(b))\right),
\end{align*}

which is positive if and only if

\begin{align*}
 b\Lambda_{1}'(b) + \left(1+\frac{\al}{2}\right)\Lambda_{1}(b) > b\Lambda_{m}'(b) + \left(1+\frac{\al}{2}\right)\Lambda_{m}(b).
\end{align*}


We now show the following identity:

\begin{lemma}
 Let $m \geq 1$. Then

\begin{align*}
 & b\Lambda_{m}'(b) + \left(1+\frac{\al}{2}\right)\Lambda_{m}(b) -\left( b\Lambda_{m+1}'(b) + \left(1+\frac{\al}{2}\right)\Lambda_{m+1}(b) \right) \\
& = \left(\frac{\left(\frac{\al}{2}\right)_{m+1}\Gamma\left(\frac{\al}{2}\right)}{m!2^{1-\al}\Gamma\left(1-\frac{\al}{2}\right)}b^{m-2}(1-b)\right)\left(\hypg\left(\frac{\al}{2},m+\frac{\al}{2},m+1,b^{2}\right) + (b-1)\hypg\left(\frac{\al}{2},m+1+\frac{\al}{2},m+1,b^{2}\right)\right)
\end{align*}

\end{lemma}

\begin{proof}
We first start with the following identity. For every $m \geq 1$:

\begin{align*}
 b\Lambda_{m}'(b) + \left(1+\frac{\al}{2}\right)\Lambda_{m}(b) & = \left(\frac{\Gamma\left(\frac{\al}{2}\right)}{2^{1-\al}\Gamma\left(1-\frac{\al}{2}\right)}\right)\frac{\left(\frac{\al}{2}\right)_{m}}{m!}(m-1)b^{m-1}\hypg\left(\frac{\al}{2},m+\frac{\al}{2},m+1,b^{2}\right) \\
& + \left(\frac{\Gamma\left(\frac{\al}{2}\right)}{2^{1-\al}\Gamma\left(1-\frac{\al}{2}\right)}\right)\frac{\left(\frac{\al}{2}\right)_{m}}{m!}2b^{m+1}\frac{\left(\frac{\al}{2}\right)\left(m+\frac{\al}{2}\right)}{m+1}\hypg\left(\frac{\al}{2}+1,m+1+\frac{\al}{2},m+2,b^{2}\right) \\
& + \left(\frac{\Gamma\left(\frac{\al}{2}\right)}{2^{1-\al}\Gamma\left(1-\frac{\al}{2}\right)}\right)\frac{\left(\frac{\al}{2}\right)_{m}}{m!}\left(1+\frac{\al}{2}\right)b^{m-1}\hypg\left(\frac{\al}{2},m+\frac{\al}{2},m+1,b^{2}\right) \\
& = \left(\frac{\Gamma\left(\frac{\al}{2}\right)}{2^{1-\al}\Gamma\left(1-\frac{\al}{2}\right)}\right)\frac{\left(\frac{\al}{2}\right)_{m}}{m!}\left(m+\frac{\al}{2}\right)b^{m-1}\hypg\left(\frac{\al}{2},m+\frac{\al}{2},m+1,b^{2}\right) \\
& + \left(\frac{\Gamma\left(\frac{\al}{2}\right)}{2^{1-\al}\Gamma\left(1-\frac{\al}{2}\right)}\right)\frac{\left(\frac{\al}{2}\right)_{m}}{m!}\left(m+\frac{\al}{2}\right)b^{m+1}\frac{\al}{m+1}\hypg\left(\frac{\al}{2}+1,m+1+\frac{\al}{2},m+2,b^{2}\right) \\
& = \left(\frac{\Gamma\left(\frac{\al}{2}\right)}{2^{1-\al}\Gamma\left(1-\frac{\al}{2}\right)}\right)\frac{\left(\frac{\al}{2}\right)_{m+1}}{m!}b^{m-1}\hypg\left(\frac{\al}{2},m+\frac{\al}{2},m+1,b^{2}\right) \\
& + \left(\frac{\Gamma\left(\frac{\al}{2}\right)}{2^{1-\al}\Gamma\left(1-\frac{\al}{2}\right)}\right)\frac{\left(\frac{\al}{2}\right)_{m+1}}{m!}b^{m+1}\frac{\al}{m+1}\hypg\left(\frac{\al}{2}+1,m+1+\frac{\al}{2},m+2,b^{2}\right) \\
\end{align*}

where we have used the expression \eqref{derivadahiper} for the derivative of the hypergeometric function. Using \eqref{hypergeometricidentity1}, we get

\begin{align*}
 b^{2}\frac{\al}{m+1}\hypg\left(\frac{\al}{2}+1,m+1+\frac{\al}{2},m+2,b^{2}\right)
& = 2\left(\hypg\left(\frac{\al}{2},m+1+\frac{\al}{2},m+1,b^{2}\right) - \hypg\left(\frac{\al}{2},m+\frac{\al}{2},m+1,b^{2}\right)\right)
\end{align*}

which implies that

\begin{align}
 b\Lambda_{m}'(b) + \left(1+\frac{\al}{2}\right)\Lambda_{m}(b) & =
\left(\frac{\Gamma\left(\frac{\al}{2}\right)}{2^{1-\al}\Gamma\left(1-\frac{\al}{2}\right)}\right) \frac{\left(\frac{\al}{2}\right)_{m+1}}{m!}b^{m-1} \nonumber \\
& \times \left(2\hypg\left(\frac{\al}{2},m+1+\frac{\al}{2},m+1,b^{2}\right)\right. \left.-\hypg\left(\frac{\al}{2},m+\frac{\al}{2},m+1,b^{2}\right)\right) \label{simplificacionblambda}
\end{align}

We now deal with the term $b\Lambda_{m+1}'(b) + \left(1+\frac{\al}{2}\right)\Lambda_{m+1}(b)$. By \eqref{simplificacionblambda}, we have that

\begin{align*}
& b\Lambda_{m+1}'(b) + \left(1+\frac{\al}{2}\right)\Lambda_{m+1}(b) \\
& =
\left(\frac{\Gamma\left(\frac{\al}{2}\right)}{2^{1-\al}\Gamma\left(1-\frac{\al}{2}\right)}\right)\frac{\left(\frac{\al}{2}\right)_{m+2}}{(m+1)!}b^{m}\left(2\hypg\left(\frac{\al}{2},m+2+\frac{\al}{2},m+2,b^{2}\right)-\hypg\left(\frac{\al}{2},m+1+\frac{\al}{2},m+2,b^{2}\right)\right) \\
& =
\left(\frac{\Gamma\left(\frac{\al}{2}\right)}{2^{1-\al}\Gamma\left(1-\frac{\al}{2}\right)}\frac{\left(\frac{\al}{2}\right)_{m+1}b^{m-1}}{m!}\right)\frac{\left(\frac{\al}{2}+m+1\right)}{(m+1)}b\left(2\hypg\left(\frac{\al}{2},m+2+\frac{\al}{2},m+2,b^{2}\right)-\hypg\left(\frac{\al}{2},m+1+\frac{\al}{2},m+2,b^{2}\right)\right) \\
\end{align*}

By \eqref{hypergeometricidentity2}, 

\begin{align*}
2\hypg\left(\frac{\al}{2},m+2+\frac{\al}{2},m+2,b^{2}\right) = \frac{2}{m+1+\frac{\al}{2}}\left((m+1)\hypg\left(\frac{\al}{2},m+1+\frac{\al}{2},m+1,b^2\right) + \frac{\al}{2}\hypg\left(\frac{\al}{2},m+1+\frac{\al}{2},m+2,b^2\right)\right),
\end{align*}

which implies

\begin{align*}
 & b\Lambda_{m+1}'(b) + \left(1+\frac{\al}{2}\right)\Lambda_{m+1}(b) \\
& =
\left(\frac{\Gamma\left(\frac{\al}{2}\right)\left(\frac{\al}{2}\right)_{m+1}b^{m-1}}{2^{1-\al}\Gamma\left(1-\frac{\al}{2}\right)m!}\right)\frac{b}{(m+1)}\left(2(m+1)\hypg\left(\frac{\al}{2},m+1+\frac{\al}{2},m+1,b^{2}\right)+\left(m+1-\frac{\al}{2}\right)\hypg\left(\frac{\al}{2},m+1+\frac{\al}{2},m+2,b^{2}\right)\right) \\
\end{align*}

Furthermore, by \eqref{hypergeometricidentity3},

\begin{align*}
 \left(m+1-\frac{\al}{2}\right)b^{2}\hypg\left(\frac{\al}{2},m+1+\frac{\al}{2},m+2,b^2\right)
= (b^{2}-1)(m+1)\hypg\left(\frac{\al}{2},m+1+\frac{\al}{2},m+1,b^2\right) + (m+1)\hypg\left(\frac{\al}{2},m+\frac{\al}{2},m+1,b^2\right).
\end{align*}

Finally, putting everything together:

\begin{align*}
& b\Lambda_{m}'(b) + \left(1+\frac{\al}{2}\right)\Lambda_{m}(b) - \left(b\Lambda_{m+1}'(b) + \left(1+\frac{\al}{2}\right)\Lambda_{m+1}(b)\right) \\
& = \left(\frac{\Gamma\left(\frac{\al}{2}\right)}{2^{1-\al}\Gamma\left(1-\frac{\al}{2}\right)}\frac{\left(\frac{\al}{2}\right)_{m+1}b^{m-1}}{m!}\frac{(1-b)}{b}\right)\left(\hypg\left(\frac{\al}{2},m+\frac{\al}{2},m+1,b^2\right) + (b-1)\hypg\left(\frac{\al}{2},m+1+\frac{\al}{2},m+1,b^2\right)\right),
\end{align*}

as we wanted to prove.

\end{proof}

The first bracket is always positive, and, since $0 < b < 1$, the second bracket can be bounded below by

\begin{align*}
 \frac{1}{b}\left(\hypg\left(\frac{\al}{2},m+\frac{\al}{2},m+1,b^{2}\right)-\hypg\left(\frac{\al}{2},m+1+\frac{\al}{2},m+1,b^{2}\right)\right) + b\hypg\left(\frac{\al}{2},m+1+\frac{\al}{2},m+1,b^{2}\right)
\end{align*}

We will focus on this term. Expanding the hypergeometric functions, we get

\begin{align*}
 & \sum_{j=1}^{\infty}\frac{1}{j!}\left(\frac{\left(\frac{\al}{2}\right)_{j}\left(m+\frac{\al}{2}\right)_{j}}{(m+1)_{j}} - \frac{\left(\frac{\al}{2}\right)_{j}\left(m+1+\frac{\al}{2}\right)_{j}}{(m+1)_{j}}\right)b^{2j-1} + \sum_{k=0}^{\infty}\frac{\left(\frac{\al}{2}\right)_{k}\left(m+1+\frac{\al}{2}\right)_{k}}{(m+1)_{k}}\frac{1}{k!}b^{2k+1} \\
& = \sum_{j=0}^{\infty}\left(\frac{1}{(j+1)!}\frac{\left(\frac{\al}{2}\right)_{j+1}}{(m+1)_{j+1}}\left(\left(m+\frac{\al}{2}\right)_{j+1}-\left(m+1+\frac{\al}{2}\right)_{j+1}\right) - \frac{1}{j!}\frac{\left(\frac{\al}{2}\right)_{j}\left(m+1+\frac{\al}{2}\right)_{j}}{(m+1)_{j}}\right)b^{2j+1} \\
& = \sum_{j=0}^{\infty}\left(\frac{1}{j!}\frac{\left(\frac{\al}{2}\right)_{j}\left(m+1+\frac{\al}{2}\right)_{j}}{(m+1)_{j}}\left(\frac{1}{j+1}\left(\frac{\frac{\al}{2}+j}{m+1+j}\right)\left(m+\frac{\al}{2} - \left(m+1+j+\frac{\al}{2}\right)\right) - 1\right) \right)b^{2j+1} \\
& = \sum_{j=0}^{\infty}\left(\frac{1}{j!}\frac{\left(\frac{\al}{2}\right)_{j}\left(m+1+\frac{\al}{2}\right)_{j}}{(m+1)_{j}}\left(\frac{m+1-\frac{\al}{2}}{m+1+j} \right)\right)b^{2j+1} > 0\\
\end{align*}

Finally, using that the sum telescopes

\begin{align*}
&  b\Lambda_{1}'(b) + \left(1+\frac{\al}{2}\right)\Lambda_{1}(b) - b\Lambda_{m}'(b) + \left(1+\frac{\al}{2}\right)\Lambda_{m}(b) \\
& = \sum_{k=1}^{m-1} \left(b\Lambda_{k}'(b) + \left(1+\frac{\al}{2}\right)\Lambda_{k}(b) - \left(b\Lambda_{k+1}'(b) + \left(1+\frac{\al}{2}\right)\Lambda_{k+1}(b)\right)\right) > 0,
\end{align*}

we conclude that $Q_{-}(b,m) \leq \Theta_m$. This finishes the proof of the proposition.

\end{proof}

In particular, this shows that if there is a solution $ 0 < b_m^{*} < 1$, then  $\Theta_{m} = Q_{+}(b_{m}^{*},m)$ has to be satisfied for some $b_{m}^{*}$ (since $\Theta_{m} = Q_{-}(b_{m}^{*},m)$ cannot hold). We now turn to the study of $Q_{+}(b,m)$ as a function of $b$. We have that:

\begin{align*}
 \lim_{b \to 1} Q_{+}(b,m) > \Theta_{m}
\end{align*}

This follows from Lemma \ref{lambdanb1}, since

\begin{align*}
 \lim_{b \to 1} Q_{+}(b,m) - \Theta_{m} = \Lambda_{m}(1) + \Lambda_1(1) - (\Lambda_{1}(1) - \Lambda_{m}(1)) = 2 \Lambda_{m}(1) > 0.
\end{align*}

Moreover,
\begin{align*}
 \lim_{b \to 0} Q_{+}(b,m) = 0,
\end{align*}

thus, by continuity, there exists $0 < b_{m}^{*} < 1$ such that $\Theta_{m} = Q_{+}(b_{m}^{*},m)$. Moreover, for that $b_{m}^{*}$, we have that

\begin{align}\label{cotathetaqplus}
 \Theta_{m} = Q_{+}(b_{m}^{*},m) > \frac12((b_{m}^{*})^{\al} + (b_{m}^{*})^{2})\Lambda_{1}(b_{m}^{*}) + \frac12((b_{m}^{*})^{\al}-(b_{m}^{*})^{2})\Lambda_{1}(b_{m}^{*}) 
= ((b_{m}^{*})^{\al})\Lambda_{1}(b_{m}^{*}) > ((b_{m}^{*})^{2})\Lambda_{1}(b_{m}^{*})
\end{align}

The next step is to show uniqueness. To do so, we will show that $Q_{+}(b,m)$ is increasing in $b$. We start considering

\begin{align*}
 \tilde{Q}_{+}(b,m) = \frac{1}{\Lambda_{1}(b)b^{\al}}Q_{+}(b,m)
= (1+b^{2-\al}) + \sqrt{(1-b^{2-\al})^{2} + 4b^{2-\al}j(b)}
\end{align*}

and we will show that $\tilde{Q}_{+}(b,m)$ is increasing in $b$. This is enough since $\Lambda_{1}(b)b^{\al}$ is an increasing function of $b$ as well. Taking a derivative with respect to $b$, one obtains:

\begin{align*}
 \pa_b \tilde{Q}_{+}(b,m) & = \frac{(2-\al)b^{1-\al}}{\sqrt{(1-b^{2-\al})^{2} + 4b^{2-\al}j(b)}} \\
& \times \left(\underbrace{\sqrt{(1-b^{2-\al})^{2} + 4b^{2-\al}j(b)} + (b^{2-\al}-1)}_{>0} + \underbrace{2j(b)}_{>0} +  \underbrace{\frac{2}{2-\al}bj'(b)}_{>0 \text{ by Lemma \ref{lemajcreciente}}}\right) > 0,
\end{align*}

as desired. Finally, we study $Q_{+}(b,m)$ as a function of $m$ and show that $b_{n}^{*} > b_{m}^{*}$ if $n > m$. This follows easily since $\Lambda_{m}(b)$ is a decreasing function of $m$ for fixed $b$. Therefore, since $\Theta_{n}$ is an increasing function of $n$, if $n > m$, then $Q_{+}(b_{m}^{*},n) < Q_{+}(b_{m}^{*},m) = \Theta_{m} < \Theta_n$ which implies $b_{n}^{*} > b_{m}^{*}$.

The one-dimensionality of the rank of $M_{m}^{\al}(b_m^{*})$ follows from the fact that $b_{m}^{*} \Lambda_m(b_{m}^{*}) \neq 0$.
\end{proof}

\begin{rem}

We remark that this approach breaks down for the 2D Euler case, where $\Delta_{m}^{0}(b) \neq 0$ for all $0 < b < 1$. Indeed, we have that

\begin{align*}
\displaystyle
M_{m}^{0}(b) = 
\left(
\begin{array}{cc}
\displaystyle
\frac{b^{2}}{2} - \frac12 + \frac{1}{2m} & 
\displaystyle- \frac{b^{m+1}}{2m} \\
\displaystyle\frac{b^{m}}{2m} & 
\displaystyle-\frac{b}{2m}
\end{array}
\right)
\end{align*}

Computing $\Delta_{m}^{0}(b)$ we obtain
\begin{align*}
\Delta_{m}^{0}(b) & = \left(\frac{b^{2}}{2} - \frac12 + \frac{1}{2m}\right)\left(-\frac{b}{2m}\right) + b\left(\frac{b^{m}}{2m}\right)^{2} = \frac{b}{4m^{2}}\left(b^{2m} - ((b^{2}-1)m+1)\right) \\
& = \frac{b(b^2-1)}{4m^{2}}\left(\frac{b^{2m} - 1}{b^2-1} - m\right) = \frac{b(b^2-1)}{4m^{2}}\left((1-1)+(b^2-1)+\ldots+(b^{2m-2}-1)\right)
\end{align*}

It is therefore clear that $\Delta_{m}^{0}(b)$ never vanishes. 
\end{rem}

\subsubsection{Codimension of the image of the linear operator.}

Let $m \geq 2$ be fixed and let $b_m^{*}$ be the value of $b$ found in Proposition \ref{propbstar}. We now characterize the image of $DF(b_m^{*},0,0)$. We have the following Lemma:

\begin{lemma}
Let

\begin{align*}
Z = \left\{(Q,q) \in Y^{k-1,m}_{c} \times Y^{k-1,m}_{c}, Q(x) = \sum_{k=1}^{\infty}Q_{km}\sin(kmx), q(x) = \sum_{k=1}^{\infty}q_{km}\sin(kmx), \right.\\ \left.\exists \lambda_{Q,q} \in \mathbb{R} \text{ s.t.} 
\left(\begin{array}{c}Q_{m} \\ q_{m}\end{array}\right)= \lambda_{Q,q} \left(
\begin{array}{c}
 -\Theta_{m} + \left(b_{m}^{*}\right)^{2}\Lambda_{1}(b_{m}^{*}) \\
b_{m}^{*} \Lambda_{m}(b_{m}^{*})
\end{array}
\right)\right\}.
\end{align*}

Then $Z = \text{Im}\left(DF(b_{m}^{*},0,0)\right)$.

\end{lemma}

\begin{proof}

We start proving that $\text{Im}\left(DF(b_{m}^{*},0,0)\right) \subset Z$. This follows easily 
since $DF$ maps

\begin{align*}
\left\{
\begin{array}{cc}
X^{k,m}_{c} \times X^{k,m}_{c} & \text{ if } \al < 1 \\
X^{k+\log,m}_{c} \times X^{k+\log,m}_{c} & \text{ if } \al = 1 \\
X^{k+\al-1,m}_{c} \times X^{k+\al-1,m}_{c} & \text{ if } \al > 1
\end{array}
\right\}
\text{ into } Y^{k-1,m}_{c} \times Y^{k-1,m}_{c}
\end{align*}

and by the explicit formula of the $m$-th mode contribution of $DF$.

We now prove the other implication and show that $Z \subset \text{Im}\left(DF(b_{m}^{*},0,0)\right)$. Let $(Q(x),q(x)) \in Z$. We want to show that there exists a 

\begin{align*}
(H(x),h(x)) \in
\left\{
\begin{array}{cc}
X^{k,m}_{c} \times X^{k,m}_{c} & \text{ if } \al < 1 \\
X^{k+\log,m}_{c} \times X^{k+\log,m}_{c} & \text{ if } \al = 1 \\
X^{k+\al-1,m}_{c} \times X^{k+\al-1,m}_{c} & \text{ if } \al > 1
\end{array}
\right\}
\end{align*}

 such that $DF(b_{m}^{*},0,0)\left[\begin{array}{cc}H\\h \end{array}\right] = \left(\begin{array}{cc}Q \\ q\end{array}\right)$. Let us project $H,h$ into Fourier modes as

\begin{align*}
H(x) = \sum_{k=1}^{\infty}H_{km}\cos(kmx), \quad h(x) = \sum_{k=1}^{\infty}h_{km}\cos(kmx).
\end{align*}

This yields the following system of equations for any $k$:

\begin{align*}
(-km) M_{km}^{\al}(b_{m}^{*})
\left(\begin{array}{c}H_{km} \\ h_{km} \end{array}\right)
= 
(-km) \left(
\begin{array}{cc}
 -\Theta_{km} + (b_{m}^{*})^{2} \Lambda_1(b_{m}^{*}) & - (b_{m}^{*})^{2}\Lambda_{km}(b_{m}^{*}) \\
b_{m}^{*} \Lambda_{km}(b_{m}^{*}) & (b_{m}^{*})^{1-\al}\Theta_{km} - b_{m}^{*} \Lambda_1(b_{m}^{*})
\end{array}
\right)
\left(\begin{array}{c}H_{km} \\ h_{km} \end{array}\right) = \left(\begin{array}{c}Q_{km} \\ q_{km} \end{array}\right),
\end{align*}

which has as solutions:

\begin{align*}
\left(\begin{array}{c}H_{km} \\ h_{km} \end{array}\right) &
= (M_{km}^{\al}(b_m^{*}))^{-1}\left(\begin{array}{c}Q_{km} \\ q_{km} \end{array}\right) \\
& = -\frac{1}{km}\frac{1}{\Delta_{km}^{\al}(b_m^*)}
 \left(
\begin{array}{cc}
(b_{m}^{*})^{1-\al}\Theta_{km} - b_{m}^{*} \Lambda_1(b_{m}^{*})  &  (b_{m}^{*})^{2}\Lambda_{km}(b_{m}^{*}) \\
-b_{m}^{*} \Lambda_{km}(b_{m}^{*}) & -\Theta_{km} + (b_{m}^{*})^{2} \Lambda_1(b_{m}^{*})
\end{array}
\right)
\left(\begin{array}{c}Q_{km} \\ q_{km} \end{array}\right)
\end{align*}

whenever $k \neq 1$ and:

\begin{align*}
H_{m} = -\frac{1}{m}\lambda_{Q,q}, \quad h_{m} = 0.
\end{align*}

Note that there are more solutions for $(H_m,h_m)$. This shows the existence of a candidate $(H,h)$. We now show that this candidate has the desired regularity. To do so, we need the following additional asymptotic Lemma:

\begin{lemma}\label{lemmaasymptoticdet}
Let $0 < \al < 2, 0 < b < 1$ and let $n \in \mathbb{Z}$. Let $\Delta_{n}^{\al}(b)$ be defined as in Proposition \ref{propbstar}, namely:

\begin{align*}
\Delta^{\al}_{n}(b) 
& = \left(b^{1-\al}\Theta_n - b \Lambda_1(b)\right)\left( -\Theta_n + b^{2} \Lambda_1(b)\right) + b^{3}\Lambda_{n}(b)^{2}
\end{align*}

Then $\Delta^{\al}_{n}(b)$ has the following asymptotic behaviour (with non-zero leading terms) as $n \to \infty$:

\begin{align*}
\Delta^{\al}_{n}(b) 
= 
\left\{
\begin{array}{cc}
\mu_{\al} + \frac{\nu_{\al}}{n^{1-\al}} + O\left(\frac{1}{n^{2-\al}}\right) & \text{ if $\al < 1$} \\
-(\log(n))^2\frac{b^{1-\al}}{\pi^2} + O(\log(n)) & \text{ if $\al = 1$} \\
\frac{p_{\al}}{n^{2-2\al}} + \frac{q_{\al}}{n^{1-\al}} + O\left(1\right) & \text{ if $\al > 1$} 
\end{array}
\right.
\end{align*}

with

\begin{align*}
\mu_{\al} = (-\Lambda_1(1) + b^2\Lambda_1(b))(b^{1-\al}\Lambda_1(1) - b\Lambda_1(b))\\
\nu_{\al} = \left(1-\frac{\al}{2}\right)\Lambda_1(1)(2b^{1-\al}\Lambda_1(1) - b(1+b^{2-\al})\Lambda_1(b))e^{\al \gamma + c_{\al}}\\
p_{\al} = -b^{1-\al}\left(\frac{\Gamma(1-\al)}{2^{1-\al}\Gamma^{2}\left(1-\frac{\al}{2}\right)}\right)^{2}\\
q_{\al} = -\frac{\Gamma(1-\al)}{2^{1-\al}\Gamma^{2}\left(1-\frac{\al}{2}\right)}b\Lambda_1(b)(1+b^{2-\al})
\end{align*}

and $\gamma, c_{\al}$ some finite constants.

\end{lemma}
\begin{proof}

We start by noticing the exponential decay in $n$ of $\Lambda_{n}(b)$ (see \cite{delaHoz-Hassainia-Hmidi:doubly-connected-vstates-gsqg}). Next, we have the asymptotic expansion for $\Theta_{n}$:

\begin{align*}
\Theta_n \sim
\left\{
\begin{array}{cc}
\Lambda_1(1) - \left(1-\frac{\al}{2}\right)\Lambda_1(1)\frac{e^{\al \gamma + c_{\al}}}{n^{1-\al}} + O(n^{\al-2}) & \text{ if $\al < 1$} \\
\frac{1}{\pi}\log(n) + O(1) & \text{ if $\al = 1$}
\end{array}
\right.
\end{align*}

which was proved in \cite{Hassainia-Hmidi:v-states-generalized-sqg} for $\al \leq 1$ we obtain $\mu_{\al}$ and $\nu_{\al}$. This shows the asymptotics for $\al \leq 1$. For $\al > 1$ it follows from the expression (see \cite{Castro-Cordoba-GomezSerrano:existence-regularity-vstates-gsqg}):

\begin{align*}
\Theta_{n} = \frac{\Gamma(1-\al)}{2^{1-\al}\Gamma^{2}\left(1-\frac{\al}{2}\right)}\left(\frac{\Gamma\left(1+\frac{\al}{2}\right)}{\Gamma\left(2-\frac{\al}{2}\right)}-\frac{\Gamma\left(n+\frac{\al}{2}\right)}{\Gamma\left(n+1-\frac{\al}{2}\right)}\right)
\end{align*}

and the asymptotic formulas for the Gamma function \cite[Formula 6.1.46, p.257]{Abramowitz-Stegun:handbook-mathematical-functions}.  

All we are left to show is that $\mu_{\al} \neq 0$ for $\al < 1$, and that $p_{\al} \neq 0$ for $\al > 1$. The former is an immediate consequence of the monotonicity in $b$ of $\Lambda_n(b)$ (Lemma \ref{lemmaA2}) and the latter is trivial.

\end{proof}

Using this Lemma, one easily obtains the following asymptotics for the inverse of $M_{km}^{\al}(b_m^{*})$ as $k \to \infty$:

\begin{corollary}\label{corinvtheta}
\begin{align*}
& (M_{km}^{\al}(b_m^{*}))^{-1}_{12} = -\frac{1}{km}\frac{1}{\Delta_{km}^{\al}(b_m^*)}(b_{m}^{*})^{2}\Lambda_{km}(b_{m}^{*}) \\
& \sim (M_{km}^{\al}(b_m^{*}))^{-1}_{21} =
-\frac{1}{km}\frac{1}{\Delta_{km}^{\al}(b_m^*)}\left(-b_{m}^{*} \Lambda_{km}(b_{m}^{*})\right) \sim O(1) \\
& (M_{km}^{\al}(b_m^{*}))^{-1}_{11} = -\frac{1}{km}\frac{1}{\Delta_{km}^{\al}(b_m^*)}\left((b_{m}^{*})^{1-\al}\Theta_{km} - b_{m}^{*} \Lambda_1(b_{m}^{*})\right) \\
& \sim 
(M_{km}^{\al}(b_m^{*}))^{-1}_{22} = -\frac{1}{km}\frac{1}{\Delta_{km}^{\al}(b_m^*)}\left( -\Theta_{km} + (b_{m}^{*})^{2} \Lambda_1(b_{m}^{*})\right) \sim
\left\{
\begin{array}{cc}
\frac{1}{k} & \text{ if $\al < 1$} \\
\frac{1}{k\log(k)} & \text{ if $\al = 1$} \\
\frac{1}{k^{\al}} & \text{ if $\al > 1$}
\end{array}
\right.
\end{align*}
\end{corollary}

We now distinguish cases depending on $\alpha$. For $\al < 1$:

\begin{align*}
\|H\|_{X^{k,m}_{c}}^{2}+\|h\|_{X^{k,m}_{c}}^{2}
& =  \sum_{j=1}^{\infty} (|H_{jm}|^{2}+|h_{jm}|^{2})(1+jm)^{2k}(\cosh(cjm)^{2} + \sinh(cjm)^{2}) \\
& = \frac{1}{m^{2}}\lambda_{Q,q}^{2}(1+m)^{2k}(\cosh(c)^{2} + \sinh(c)^{2}) \\
& + \sum_{j=2}^{\infty} (1+jm)^{2k}(\cosh(cjm)^{2} + \sinh(cjm)^{2})\\
& \times \left\{[(M_{jm}^{\al}(b_m^{*}))^{-1}_{11}Q_{mj} + (M_{jm}^{\al}(b_m^{*}))^{-1}_{12}q_{mj}]^{2}+
[(M_{jm}^{\al}(b_m^{*}))^{-1}_{21}Q_{mj} + (M_{jm}^{\al}(b_m^{*}))^{-1}_{22}q_{mj}]^{2}
\right\} \\
& \leq C + C\|Q\|_{Y^{k-1,m}_{c}}^{2} + C\|q\|_{Y^{k-1,m}_{c}}^{2} < \infty,
\end{align*}

where in the last line we have used corollary \ref{corinvtheta}. For $\al > 1$ and $\al = 1$, one obtains using the same approach and the asymptotics from corollary \ref{corinvtheta}:

\begin{align*}
\|H\|_{X^{k+\al-1,m}_{c}}^{2}+\|h\|_{X^{k+\al-1,m}_{c}}^{2}
& \leq C + C\|Q\|_{Y^{k-1,m}_{c}}^{2} + C\|q\|_{Y^{k-1,m}_{c}}^{2} < \infty \\
\|H\|_{X^{k+\log,m}_{c}}^{2}+\|h\|_{X^{k+\log,m}_{c}}^{2}
& \leq C + C\|Q\|_{Y^{k-1,m}_{c}}^{2} + C\|q\|_{Y^{k-1,m}_{c}}^{2} < \infty,
\end{align*}

respectively. This concludes that $Z = \text{Im}\left(DF(b_{m}^{*},0,0)\right)$ and in particular shows that the codimension of the image of $DF(b_{m}^{*},0,0)$ is 1, as we needed.

\end{proof}

\subsection{Step 5}

This step is devoted to show the transversality condition. We start writing out the calculations since everything is explicit, including the characterization of the image done in the previous subsection. Based on that, we have the following:

\begin{align*}
 \pa_{b} M_{m}^{\al}(b_{m}^{*}) & = \left(
\begin{array}{cc}
 (b_{m}^{*})^{2} \Lambda_1'(b_{m}^{*}) + 2b_{m}^{*} \Lambda_{1}(b_{m}^{*}) & - (b_{m}^{*})^{2}\Lambda_{n}'(b_{m}^{*}) - 2b_{m}^{*} \Lambda_{n}(b_{m}^{*}) \\
b_{m}^{*} \Lambda_n'(b_{m}^{*}) + \Lambda_{n}(b_{m}^{*}) & (1-\al)(b_{m}^{*})^{-\al}\Theta_n - b_{m}^{*} \Lambda_1'(b_{m}^{*}) - \Lambda_{1}(b_{m}^{*})
\end{array}
\right)
\end{align*}

Letting 

\begin{align*}
 v_{0}(b_{m}^{*}) = 
\left(
\begin{array}{c}
 (b_{m}^{*})^{2} \Lambda_{n}(b_{m}^{*}) \\
(b_{m}^{*})^{2}\Lambda_{1}(b_{m}^{*}) - \Theta_{n}
\end{array}
\right), \quad
 w(b_{m}^{*}) = 
\left(
\begin{array}{c}
 -\Theta_{n} + (b_{m}^{*})^{2}\Lambda_{1}(b_{m}^{*}) \\
b_{m}^{*} \Lambda_{n}(b_{m}^{*})
\end{array}
\right), 
\end{align*}

be the generators of Ker$(M_{m}^{\al}(b_{m}^{*}))$ and Im$(M_{m}^{\al}(b_{m}^{*}))$ respectively, the transversality condition is equivalent to prove that $w_1(b_{m}^{*})$ and $w(b_{m}^{*})$ are not parallel, where

\begin{align*}
w_{1}(b_{m}^{*}) & =  \pa_{b} M_{m}^{\al}(b_{m}^{*}) v_{0}(b_{m}^{*}) \\
& = 
\left(
\begin{array}{c}
(2b_{m}^{*}\Lambda_{1}(b_{m}^{*}) + (b_{m}^{*})^{2}\Lambda_{1}'(b_{m}^{*}))(b_{m}^{*})^{2}\Lambda_{n}(b_{m}^{*}) - (2b_{m}^{*}\Lambda_{n}(b_{m}^{*}) + (b_{m}^{*})^{2}\Lambda_{n}'(b_{m}^{*}))((b_{m}^{*})^{2}\Lambda_{1}(b_{m}^{*}) - \Theta_n) \\
(\Lambda_{n}(b_{m}^{*}) + b_{m}^{*}\Lambda_{n}'(b_{m}^{*}))(b_{m}^{*})^{2}\Lambda_{n}(b_{m}^{*}) + ((1-\al)(b_{m}^{*})^{-\al}\Theta_{n} - \Lambda_{1}(b_{m}^{*}) - b_{m}^{*}\Lambda_{1}'(b_{m}^{*}))((b_{m}^{*})^{2}\Lambda_{1}(b_{m}^{*}) - \Theta_{n})
\end{array}
\right)
\end{align*}

In order to do so, we claim that both components of $w_{1}(b_{m}^{*})$ have the same (positive) sign, whereas the two components of $w(b_{m}^{*})$ have opposite signs. The latter is easy to establish and follows from Lemma \ref{lambdanb1} and \eqref{cotathetaqplus}. We focus on showing that both components of $w_{1}(b_{m}^{*})$ are positive. The first one is equal to

\begin{align*}
 & (b_{m}^{*})^{4}\Lambda_{1}'(b_{m}^{*})\Lambda_{m}(b_{m}^{*}) - (b_{m}^{*})^{4}\Lambda_{1}(b_{m}^{*})\Lambda'_{m}(b_{m}^{*}) + \Theta_{m}(2b_{m}^{*}\Lambda_{m}(b_{m}^{*}) + (b_{m}^{*})^{2}\Lambda_{m}'(b_{m}^{*})) \\
& > (b_{m}^{*})^{4}\Lambda_{1}'(b_{m}^{*})\Lambda_{m}(b_{m}^{*}) - (b_{m}^{*})^{4}\Lambda_{1}(b_{m}^{*})\Lambda'_{m}(b_{m}^{*}) + (b_{m}^{*})^{2}\Lambda_{1}(b_{m}^{*})(2b_{m}^{*}\Lambda_{m}(b_{m}^{*}) + (b_{m}^{*})^{2}\Lambda_{m}'(b_{m}^{*})) \\
& = (b_{m}^{*})^{4}\Lambda_{1}'(b_{m}^{*})\Lambda_{m}(b_{m}^{*}) + 2(b_{m}^{*})^{3}\Lambda_{1}(b_{m}^{*})\Lambda_{m}(b_{m}^{*})  > 0,
\end{align*}

and the second one is

\begin{align*}
& (b_{m}^{*})^{2}\Lambda_{m}(b_{m}^{*})^{2} + (b_{m}^{*})^{3}\Lambda_{m}'(b_{m}^{*}) + (1-\al)(b_{m}^{*})^{2-\al}\Theta_{m}\Lambda_{1}(b_{m}^{*}) - (1-\al)(b_{m}^{*})^{-\al}\Theta_{m}^{2} \\
 & - (b_{m}^{*})^{2}\Lambda_{1}(b_{m}^{*})^{2} + \Lambda_{1}(b_{m}^{*})\Theta_{m} - (b_{m}^{*})^{3}\Lambda_{1}'(b_{m}^{*})\Lambda_{1}(b_{m}^{*}) + b_{m}^{*}\Lambda_{1}'(b_{m}^{*})\Theta_{m} \\
& = \al[ (b_{m}^{*})^{2}(\Lambda_{m}(b_{m}^{*})^{2} - \Lambda_{1}(b_{m}^{*})^{2}) + \Theta_{m} \Lambda_{1}(b_{m}^{*})] + (b_{m}^{*})^{3}\Lambda_{m}'(b_{m}^{*}) + [b_{m}^{*}\Lambda_{1}'(b_{m}^{*})\Theta_{m} - (b_{m}^{*})^{3} \Lambda_{1}(b_{m}^{*})\Lambda_{1}'(b_{m}^{*})]
\end{align*}

where we have used that

\begin{align*}
 -(b_{m}^{*})^{-\al}\Theta_{m}^{2} + \Theta_{m}(b_{m}^{*})^{2-\al}\Lambda_{1}(b_{m}^{*}) = - \Theta_{m}\Lambda_{1}(b_{m}^{*}) - (b_{m}^{*})^{2}(\Lambda_{m}(b_{m}^{*})^{2} - \Lambda_{1}(b_{m}^{*})^{2}).
\end{align*}

Both square brackets are positive by \eqref{cotathetaqplus}, and the claim follows.

\subsection{Step 6}

This follows easily by doing the change of variables $y \rightarrow -y$ and $y \rightarrow y + \frac{2\pi}{m}$ inside the integral operators.

\section*{Acknowledgements}

 J.G.-S. was partially supported by the grant MTM2014-59488-P (Spain), by the ICMAT-Severo Ochoa grant SEV-2015-0554, by the Simons Collaboration Grant 524109 and by the NSF-DMS 1763356 Grant. We would like to thank Angel Castro and Diego C\'ordoba for useful discussions.

\appendix

\section{Hypergeometric function identities}

Here we collect a few facts about $\Theta_{m}$ and $\Lambda_{m}$, and about hypergeometric functions that will be used along the proofs. Recall that $\Theta_{m}$ and $\Lambda_{m}$ were defined in \eqref{defilambdan} by:

\begin{align*}
 \Lambda_n(b) & \equiv \frac{1}{b} \int_{0}^{\infty} \frac{1}{t^{1-\al}}J_n(bt)J_n(t) dt \\
& = \frac{\Gamma\left(\frac{\al}{2}\right)}{\Gamma\left(1-\frac{\al}{2}\right)2^{1-\al}} \frac{\left(\frac{\al}{2}\right)_{n}}{n!}b^{n-1} F\left(\frac{\al}{2}, n+\frac{\al}{2},n+1,b^{2}\right), \\
& = \frac{b^{n-1}}{2^{1-\al}\Gamma\left(1-\frac{\al}{2}\right)^{2}} \int_{0}^{1} x^{n-1+\frac{\al}{2}}(1-x)^{-\frac{\al}{2}}(1-b^{2}x)^{-\frac{\al}{2}}dx. \\
\Theta_{n} & \equiv \Lambda_{1}(1) - \Lambda_{n}(1)
\end{align*}

\begin{lemma}
 We have the following identities for the hypergeometric function:

\begin{eqnarray}\label{derivadahiper}
 \frac{\partial}{\partial x} \hypg(a,b,c,x) = \frac{ab}{c}\hypg(a+1,b+1,c+1,x) \\
\label{hypergeometricidentity1}
 c\hypg(a,b,c,z) - c\hypg(a,b+1,c,z) + az\hypg(a+1,b+1,c+1,z) = 0 \\
 \label{hypergeometricidentity2}
 c\hypg(a,b,c,z) - (c-b)\hypg(a,b+1,c,z) - b\hypg(a,b+1,c+1,z) = 0 \\
 \label{hypergeometricidentity3}
 c\hypg(a,b-1,c,z) + (a-c)\hypg(a,b,c+1,z) + (z-1)c\hypg(a,b,c,z) = 0 \\
 \label{312hmidi}
c\hypg(a,b,c,z) - c\hypg(a+1,b,c,z) + bz\hypg(a+1,b+1,c+1,z) = 0 \\
\label{316hmidi}
b\hypg(a,b+1,c,z) - a\hypg(a+1,b,c,z) + (a-b)\hypg(a,b,c,z) = 0
\end{eqnarray}

\end{lemma}

\begin{proof}
See \cite{Rainville:special-functions}.
\end{proof}

\begin{lemma}\label{lambdanb1}
$\Lambda_n(b)$ is an increasing function of $b$, it satisfies $\Lambda_n(b) \geq 0$ for any $n \geq 1$ and $b \in (0,1]$, and

\begin{align*}
 \lim_{b\to 1}\Lambda_{n}(b) > 0.
\end{align*}

\end{lemma}
\begin{proof}
 This follows from the integral formula \eqref{defilambdan}.
\end{proof}

\begin{lemma}
 Let $\al \in (0,2)$ and $n \geq 2$. Then:
\begin{align*}
 \Lambda_{n}(b) < \Lambda_{1}(b)
\end{align*}

for all $b \in (0,1)$.

\end{lemma}

\begin{proof}
 The proof can be found in \cite[Lemma 5.2(1)]{delaHoz-Hassainia-Hmidi:doubly-connected-vstates-gsqg}.
\end{proof}

\section{Basic integrals}

The following two lemmas will deal with the integrals that appear throughout the calculation of the linear operator:

\begin{lemma}\label{lemaA1}
Let $0 < b < 1, 0 < \al < 2, m \in \mathbb{N}$. We have that:
\begin{align*}
\frac{1}{2\pi}\int_{0}^{2\pi}\frac{\cos(my)}{(1+b^2-2b\cos(y))^{\al/2}}dy
= b^{m}\frac{\left(\frac{\al}{2}\right)_{m}}{m!} \hypg\left(\frac{\al}{2},m+\frac{\al}{2}; m+1; b^2\right)
\end{align*}
\end{lemma}
\begin{proof}
See \cite[Lemma 3.2, Eq. (3.19)]{delaHoz-Hassainia-Hmidi:doubly-connected-vstates-gsqg}: their proof can be extended to the case $0 < \al < 2$.
\end{proof}

\begin{lemma}\label{lemmaA2}
Let $0 < b < 1, 0 < \al < 2, m \in \mathbb{N}$. We have that:
\begin{align*}
\frac{1}{m}\frac{\al}{2}\int_{0}^{2\pi}\frac{2\sin(y)\sin(my)}{(1+b^2-2b\cos(y))^{\al/2+1}}dy
= \frac{b^{m-1}\left(\frac{\al}{2}\right)_{m}}{m!}\hypg\left(\frac{\al}{2}, m +\frac{\al}{2}, m+1; b^2\right)
\end{align*}
\end{lemma}
\begin{proof}
Using the trigonometric addition formulas and Lemma \ref{lemmaA2}, the LHS is equal to

\begin{align*}
\frac{b^{m-1}}{m!}\left(\frac{\al}{2}\right)_{m}\left(
\hypg\left(\frac{\al}{2}+1,\frac{\al}{2}+m,m;b^{2}\right)
 - \frac{b^{2}}{m(m+1)}\left(\frac{\al}{2}+m\right)\left(\frac{\al}{2}+m+1\right)\hypg\left(\frac{\al}{2}+1,\frac{\al}{2}+m+2,m+2;b^2\right)
\right)
\end{align*}

Combining formulas \eqref{312hmidi} with $a = \frac{\al}{2}, b = \frac{\al}{2} + m + 1, c = m+1$, \eqref{hypergeometricidentity2} with $a = \frac{\al}{2}+1, b = \frac{\al}{2}+m, c = m$ and \eqref{316hmidi} with $a = \frac{\al}{2}, b = \frac{\al}{2}+m, c = m+1$ yields the result.

\end{proof}

 \bibliographystyle{abbrv}
 \bibliography{references}

 \begin{tabular}{l}
 \textbf{Javier G\'omez-Serrano} \\
 {\small Department of Mathematics} \\
 {\small Princeton University}\\
 {\small 610 Fine Hall, Washington Rd,}\\
 {\small Princeton, NJ 08544, USA}\\
  {\small Email: jg27@math.princeton.edu} \\
   \\
 
 \end{tabular}

\end{document}